\documentclass[imslayout,noinfoline]{imsart}
\usepackage[OT1]{fontenc}
\usepackage{amssymb, amsfonts, amsmath, amsthm, graphics, epsfig}
\usepackage[colorlinks]{hyperref}

\startlocaldefs

\numberwithin{equation}{section}
\theoremstyle{plain}

\newtheorem{theorem}{Theorem}[section]
\newtheorem{prop}{Proposition}[section]
\newtheorem{lemma}{Lemma}[section]

\newtheorem{example}{Example}[section]

\newtheorem{procedure}{Procedure}[section]

\long\def\comment#1{{}}

\def\Grp#1{\left(#1\right)}
\def\Cbr#1{\left\{#1\right\}}
\def\Sbr#1{\left[#1\right]}

\def\Abs#1{\left|#1\right|}
\def\Sb #1{\sb{({#1})}}
\def\nth#1{\frac{1}{#1}}
\def\cf#1{\mathbf{1}\Cbr{#1}}
\def\tp{\sp{\top}}
\def\vf#1{\boldsymbol{#1}}
\def\eno#1#2{{#1}_1, \ldots, {#1}_{#2}}

\def\vfc{{\vf c}}
\def\vfd{{\vf d}}
\def\vff{{\vf f}}
\def\vfr{{\vf r}}
\def\vfu{{\vf u}}
\def\vfv{{\vf v}}
\def\vfx{{\vf x}}

\def\CA{\mathcal{A}}

\def\CF{\mathcal{F}}
\def\CI{\mathcal{I}}
\def\CT{\mathcal{T}}
\def\bbG{\mathbb{G}}
\def\Reals{\mathbb{R}}
\def\Rats{\mathbb{Q}}

\def\pa{a}

\def\pt{\nu}

\def\pn{\mu}

\def\PA{G}

\def\PT{F}

\def\PD{Q}

\def\EP{\mathbb{F}}

\def\dto{\downarrow}
\def\uto{\uparrow}
\def\toi{\to\infty}
\def\convP{\stackrel{\rm P}{\to}}
\def\Linf{\mathop{\underline{\mathrm{lim}}}}
\def\Lsup{\mathop{\overline{\mathrm{lim}}}}

\def\pfdr{\mathrm{pFDR}}
\def\fdr{\mathrm{FDR}}

\def\unif{\mathrm{Unif}}
\def\sgn{\text{sgn}}

\def\eqed{\makebox[.1ex]{\rule{.15\textwidth}{0ex}\qed}}

\def\pmax#1{p_{#1,\text{max}}}
\def\pmix#1{p_{#1,\text{mix}}}
\def\pglb#1{p_{#1,\text{glb}}}
\def\pseq#1{p_{#1,\text{seq}}}

\def\NE#1{\times 10^{#1}}

\endlocaldefs

\begin{document}
\begin{frontmatter}
  \title{FDR control for multiple hypothesis testing on composite
    nulls}

  \runtitle{FDR control for composite nulls}

  \begin{aug}
    \author{
      \fnms{Zhiyi}
      \snm{Chi}
      \thanksref{nih}
      \ead[label=email]{zchi@stat.uconn.edu}
    }
    \thankstext{nih}{
      Research partially supported by NSF DMS~0706048 and NIH
      MH~68028.
    }

    \runauthor{Z. Chi}
    
    \affiliation{
      Department of Statistics, University of Connecticut
    }

    \address{
      Department of Statistics \\
      University of Connecticut \\
      215 Glenbrook Road, U-4120 \\
      Storrs, CT 06269 \\
      \printead{email}
    }
  \end{aug}

  \begin{abstract}
    Multiple hypothesis testing often involves composite nulls, i.e.,
    nulls that are associated with two or more distributions.  In many
    cases, it is reasonable to assume that there is a prior
    distribution on the distributions despite it is unknown.
    When the number of distributions under true nulls is finite, we
    show that under the above assumption, the false discover rate
    (FDR) can be controlled using $p$-values computed under
    constraints imposed by the empirical distribution of the
    observations.  Comparing to FDR control using $p$-values defined
    as maximum significance level over all null distributions, the
    proposed FDR control can have substantially more power.
  \end{abstract}


  
\end{frontmatter}

\section{Introduction} \label{sec:intro}
In hypothesis testing, a relatively simple case is where the data
associated with true nulls and those with false nulls each follow a
common distribution (``simple \emph{versus\/} simple'')
\cite{efron:08, genovese:was:02}.   On the other hand, in many cases,
either the data associated with true nulls follow different
distributions (``composite nulls'') or those associated with false
nulls follow different distributions (``composite alternatives'').  In
the current literature on multiple testing, once appropriate test
statistics such as $p$-values are computed, testing procedures based
on the statistics usually do not distinguish between the simple and
composite cases \cite{lehmman:etal:05, lehmann:rom:05, vanderlaan:04a,
  genovese:was:06, sarkar:08}.  At the time when a procedure is
applied, it only has the test statistics available.  For this reason,
how the test statistics are defined plays an important role in the
overall performance of the procedure.

For composite nulls, $p$-values are usually defined as maximum
probabilities over all null distributions \cite{lehmann:rom:05}.
Following the random-effects extension for composite alternatives
\cite{genovese:was:02}, a Bayesian approach to calculating $p$-values
can be used.  Specifically, one assumes that there is a known prior
distribution on the null distributions.  Since the overall
distribution of the data associated with true nulls can now be
determined by an integral of the null distributions weighted by the
prior, the composite case is essentially reduced to the simple one.

The focus of the article lies between the above two approaches.  The
underlying premise is that there is a prior distribution on the null
distributions, however, the prior is unknown.  The basic observation
is that, in the presence of a large number of nulls, the empirical
distribution of the data provides useful information on the prior.
More specifically, the mixture of the null distributions, if
multiplied by the population fraction of true nulls, is dominated by
the empirical distribution of the data plus a small margin.  This
constrains the set of possible priors.  We shall explore the
observation for the case where there are only a finite number of null
distributions.  On the one hand, the $p$-values will be calculated as
maximum probabilities.  On the other, the maximization is over a range
of linear combinations of the null distributions, with the
coefficients being constrained.  As a result, the $p$-values can be
computed by linear programming.

The article does not consider the case of composite alternatives.  The
position here is that, since oftentimes no information on the
distributions under false nulls is available, it is sensible to regard
data associated with false nulls as being sampled from a single
overall distribution.

Although our focus is the evaluation of $p$-values under constraints,
we start with Section \ref{sec:max} on FDR control using maximum
probabilities without constraints.  That the BH procedure can control
the FDR in this case is known \cite{benjamini:yek:01}.  The purpose of
the section is to setup suitable framework for following sections, by
making a more general description of the BH procedure and indicating
where constrained maximization may be introduced.

Section \ref{sec:lp-bh} considers two ways to compute $p$-values.  The
first one is sequential, such that the $p$-value of each observation
is obtained under linear constraints imposed by observations whose
$p$-values have already been computed.  In the second one, in
principle, the $p$-values can be computed for the observations
simultaneously under the linear constraints imposed by the entire
data.  Both types of $p$-values are then processed by the BH
procedure.  Analytically, it is easier to establish FDR control based
on the first type of $p$-values because the sequential computation
allows one to use a stopping time argument \cite{storey:tay:sie:04}.
On the other hand, since there are more constraints imposed on the
second type of $p$-values, presumably they may lead to more
improvement in multiple testing.  However, the simulation study
reported in Section \ref{sec:simul} indicates that the two types of
$p$-values lead to similar performance of multiple testing.
Some possible explanations for this will be given at the end of
Section \ref{sec:simul}.  The study shows that, the BH procedure is
substantially more powerful when using the two types of $p$-values
than using $p$-values computed by the usual unconstrained
maximization.  In addition to power, we will also compare the FDR and
positive FDR (pFDR) realized by the $p$-values.

The results in Section \ref{sec:simul} indicate that in general, for
the case of composite nulls, the prior on the null distributions
cannot be estimated consistently.  Basically, this is because the
constraints imposed by the data cannot yield exact details of the
prior and also because the above two ways to evaluate $p$-values
usually select different linear combinations of the null
distributions for different observations.  This is in contrast to the
simple case, where the fraction of true nulls can be estimated
consistently \cite{benjamini:hoc:00, jin:cai:07, storey:tay:sie:04}.
Conceptually it is of interest to ask whether there are conditions
that allow the prior of the null distributions to be
estimated consistently.  In Section \ref{sec:mle}, for the case where
there are only a finite number of null distributions, a necessary
and sufficient condition will be given for the consistent estimation
of the prior using maximum likelihood estimation (MLE).  Note that, in
the MLE, the distribution under false nulls is unknown, and the
data are treated as though all are sampled from true nulls.  An
example will be given to show that for any finite set of linearly
independent null distributions, one can construct a large class of
distributions that satisfy the condition.

Section \ref{sec:discussion} contains a brief discussion.  Most
technical details are collected in the Appendix.

\subsection{Assumptions and notation} 
Let $\{\PT_\theta, \,\theta\in\Theta\}$ be a family of distributions
on $\Reals^d$.  Given random observations $\eno X n\in\Reals^d$, the
composite nulls to be tested are
\begin{align*}
  H_i: X_i\sim \PT_\theta \text{ for some } \theta\in\Theta.
\end{align*}
Each $\PT_\theta$ is a null distribution.

Our discussion will be under the following random mixture model.  The
distribution under false nulls is $\PA\not\in \{\PT_\theta,\,
\theta\in \Theta\}$ and the fraction of false nulls among all nulls is
$\pa\in (0,1)$.  There is a prior probability measure $\pt$ on
$\Theta$.  The data are sampled as follows.  Define probability
measure $\pn$ on $\Theta\cup\{*\}$, where $*$ is any element not in
$\Theta$, such that $\pn(\{*\})=\pa$ and $\pn(A) = (1-\pa) \pt(A)$ for
$A\subset \Theta$.  Sample $\eno\eta n$ iid $\sim \pn$.  If $\eta_i =
*$, then sample $X_i\sim \PA$; otherwise, sample $X_i \sim
\PT_{\eta_i}$.  Thus $\eta_i$ can be thought of as the identity of
$X_i$, indexing the distribution $X_i$ is sampled from.

Throughout we will make two assumptions.  First, $\pt$ is unknown.
Indeed, if $\pt$ is known, then under true $H_i$, $X_i \sim \PT
= \int \PT_\theta \pt(d\theta)$ and thus the composite null can be
reduced to a simple null.  Second, $\PA$ is unknown. 
This assumption is especially intended for the case where $\Theta$ is
finite.  Indeed, if $\PA$ is known, then for $n\gg 1$, both $\pa$ and
$\pt$ can be estimated accurately by the MLE, which reduces the
testing problem into one only involving simple nulls.

Recall that for a multiple testing procedure, if $R$ is the number of
rejected nulls, and $V$ that of rejected true nulls, then
\begin{align*}
  \fdr = E\Sbr{\frac{V}{R\vee 1}}, \quad
  \pfdr = E\Sbr{\frac{V}{R}\,\vline\, R>0}.
\end{align*}
Furthermore, if there are $n$ nulls and $N$ of them are true, then
\begin{align*}
  \text{power} = E\Sbr{\frac{R-V}{(n-N)\vee 1}}.
\end{align*}

\section{Testing based on maximum probabilities}
\label{sec:max}
Usually, a description of multiple testing procedure starts with
$p$-values, treating them as already available.  For our discussion
later, it is useful to start with how $p$-values are computed.  The
$p$-values are absent in the continuous version of our description,
but explicit in the discrete version.

Let $\{D_t: t\in \CI\}$ be a family of Borel sets in $\Reals^d$ 
satisfying the following conditions, where $\CI\not=\emptyset$ is an
open interval in $\Reals$.
\begin{enumerate}
\item[]\hspace{-1.2em}
  D1.  The family is increasing and right-continuous, i.e.
  $D_t = \bigcap_{s>t, s\in \CI} D_s$, for $t\in \CI$.
\item[]\hspace{-1.2em}
  D2.  $\bigcup_{t\in\CI} D_t = \Reals^d$.
\item[]\hspace{-1.2em}
  D3.  $\PA(\bigcap_{t\in\CI} D_t)=\PT_\theta(\bigcap_{t\in\CI}
  D_t)=0$, $\theta\in\Theta$.
\end{enumerate}

For each $\theta\in\Theta$, define
\begin{align} \label{eq:prob-region}
  \phi_\theta(t)=
  \begin{cases}
    \PT_\theta(D_t) & \text{if } t\in \CI, \\
    0 & \text{if } t\le \inf\CI, \\
    1 & \text{if } t\ge \sup\CI,
  \end{cases}
\end{align}
i.e., $\phi_\theta(t)$ is the significance level of the region $D_t$
under $\PT_\theta$.  By D2 and D3, $\phi_\theta$ is nondecreasing and
continuous at $\inf\CI$ and $\sup\CI$.  Denote
\begin{gather}
  M(t) = \sup_\theta \phi_\theta(t),
  \label{eq:Max}
\end{gather}
i.e., $M(t)$ is the significance level of $D_t$ associated with
$\{\PT_\theta,\, \theta\in\Theta\}$.  It is nondecreasing with
$M(t)=0$ for $t\le \inf\CI$ and $M(t)=1$ for $t\ge \sup\CI$.  

We can regard $M(t)$ as $\sup_\mu \int \phi_\theta(t)\,d\mu(\theta)$,
where the supremum is taken over all possible probability measures
$\mu$ on $\Theta$.  By our assumption, there is a prior $\pt$ on
$\Theta$.  If there is no information on the value of $\pt$, then the
supremum is justified.  If, on the other hand, it is known that $\pt$
satisfies certain conditions, then it makes sense to use the
conditions to constrain the supremum, even though the conditions may not
uniquely determine $\pt$.  This may yield a $M(t)$ closer to $\int
\phi_\theta(t)\,d\pt(\theta)$ that improves the performance 
of multiple testing.

Once $M(t)$ are in place, the BH procedure can be applied.  The
procedure can be described in two ways.  The continuous version
features a stopping time that may simplify the analysis of FDR control
(cf.\ \cite{storey:tay:sie:04}), while the discrete one is easier to
implement.  For $t\in \CI$, denote
\begin{gather*}
  R_n(t) = \sum_{i=1}^n \cf{X_i \in D_t}, \ \
  V_n(t) = \sum_{i=1}^n \cf{X_i \in D_t,\, \eta_i\in\Theta}.
\end{gather*}

\begin{procedure}[Continuous version] \rm \label{proc:max}
  Given control parameter $\alpha \in (0,1)$, let
  \begin{align*}
    \CI_R = \Cbr{t\in \CI: \frac{M(t)}{\alpha} \le
      \frac{R_n(t) \vee 1}{n}
    }.
  \end{align*}
  If $\CI_R\not=\emptyset$, set $\tau=\inf\CI_R$ and reject $H_i$
  if and only if $X_i \in D_\tau$.  Otherwise, set $\tau=\inf\CI$ and
  accept all $H_i$.  \qed
\end{procedure}

To describe the discrete version of Procedure \ref{proc:max}, define
\begin{align}\label{eq:s}
  s(x) = \inf\{t\in\CI: x\in D_t\}, \quad s_i = s(X_i), \
  i=1,\ldots,n.
\end{align}
By D2, the set in \eqref{eq:s} is nonempty, so $s(x)$ is
well-defined and $s(x) <\sup\CI$.

\begin{prop}\label{prop:equivalence}
  Under D1-3, the following statements hold.
  \begin{enumerate}
  \item[1)] $s_i\in \CI$ almost surely.
  \item[2)] For any $t\in\CI$, $s_i\le t \iff X_i\in D_t$ and hence
    $R_n(t)= \sum \cf{s_i\le t}$.
  \item[3)] Given $\theta$, if $X_i\sim \PT_\theta$, then
    $s_i\sim \phi_\theta$.
  \item[4)] For $i=1,\ldots, n$, the distribution function of $s_i$ is
    \begin{align*}
      \PD(t)=(1-\pa)\int \phi_\theta(t)\,\pt(d\theta) + \pa \PA(D_t).
    \end{align*}
  \item[5)] If $\phi_\theta\in C(\Reals)$ for all $\theta$, then
    $M(t)$ is left-continuous.
  \end{enumerate}
\end{prop}

By Proposition \ref{prop:equivalence}, $\phi_\theta(s_i)$ is the
$p$-value of $X_i$ under $\PT_\theta$.  Therefore, $M(s_i)$ can be
used as a $p$-value under the composite null $H_i$
\cite{lehmann:rom:05}.

\begin{procedure}[Discrete version] \rm \label{proc:equivalent}
  Let $s\Sb 1 \le \ldots \le s\Sb n$ be the order statistics
  of $s_i$ and $s\Sb 0=\inf\CI$.  Reject $H_i$ if and only
  $s_i\le s\Sb R$, where 
  \begin{align*}
    R = \max\Cbr{i\ge 0: \frac{M(s\Sb i)}{\alpha} \le \frac{i}{n}}.
    \eqed
  \end{align*}
\end{procedure}

\begin{prop} \label{prop:FDR}
  Suppose $\phi_\theta\in C(\Reals)$ for all $\theta$.  Then 
  Procedures \ref{proc:max} and \ref{proc:equivalent} are the same,
  and both have $\fdr \le (1-\pa) \alpha$.
\end{prop}

In single hypothesis tests, nested rejection regions are usually
indexed by significance level.  For FDR control, other indices can be
used.  This allows one to think about the rejection regions in more
natural terms and also avoids problems when different regions have the
same significance levels.

\begin{example}\label{ex:p-value}\rm
  Suppose $X_i\in \Reals$.  To use lower-tail probabilities as
  $p$-values, set $D_t = (-\infty,t]$, $t\in\CI=\Reals$.  Then
  $s_i=X_i$ and $\phi_\theta(s_i) = \PT_\theta(X_i)$.  To use
  upper-tail probabilities as $p$-values, set $D_t = [-t, \infty)$,
  $t\in\CI = \Reals$.  Then  $s_i=-X_i$ and $\phi_\theta(s_i) =
  \PT_\theta([-s_i,\infty)) = \PT_\theta([X_i,\infty))$.  Suppose each
  $\PT_\theta$ is continuous at 0.  If we use $D_t = [-t,t]$, $t\in
  \CI=[0,\infty)$, then $s_i=|X_i|$ and $\phi_\theta(s_i) = F([-|X_i|,
 |X_i|])$.  \qed 
\end{example}

\section{Testing based on constrained maximum probabilities}
\label{sec:lp-bh}
\subsection{Outlines}
Testing using maximum probabilities can be very conservative.  Our
goal is to find alternative methods when $\Theta$ is a finite set
$\{\theta_k,\ k=1,\ldots, L\}$.  The probability measure $\pt$ on
$\Theta$ can now be specified by $\vf\pt = (\eno\pt L)\tp$ with $\pt_k
= \pt(\{\theta_k\})$.  Henceforth, a letter in boldface will stand for
an $L$-dimensional vector.  Denote $\phi_k(t) = \phi_{\theta_k}(t)$.
In this section, we assume that all $\PT_k$ and hence all $\phi_k(t)$
are continuous.  Denote
\begin{gather*}
  \EP_n(t) = R_n(t)/n,
\end{gather*}
i.e. the empirical distribution based on $\eno s n$ defined in
\eqref{eq:s}.

Instead of $M(t)=\max_k \phi_k(t)$ as in Procedure \ref{proc:max}, for
finite $\Theta$, the proposed functions to use have the general form 
\begin{align*}
  M_n(t) = \sup\{
    \vfc\tp \vf\phi(t): \vfc\in C, \ \vfc\tp \vf\phi \in \CA_{n,t}
  \},
\end{align*}
where $C$ is a suitable subset of
\begin{align*}
  \Delta = \{\vfc\in [0,1]^L: \ c_1+\cdots+c_L\le 1\}
\end{align*}
and for $n\ge 1$ and $t\in \CI$, $\CA_{n,t}$ is a family of functions
on $\CI$.  In general, $C$ is constructed based on deterministic
knowledge on $\vf\pt$ and $\pa$.  On the other hand, $\CA_{n,t}$ is
constructed based on the data and hence both $M_n(t)$ and $\CA_{n,t}$
may be random.  If $C=\Delta$ and $\CA_{n,t}$ is the entire family of
functions on $\CI$, then $M_n(t)$ is $\max_i\phi_i(t)$ and we recover
Procedure \ref{proc:max}.  By adding conditions to make $C$ or
$\CA_{n,t}$ smaller, $M_n(t)$ can be smaller than $\max_i\phi_i(t)$, 
which may result in higher power.  In particular, if $C=\{\vf\pt\}$,
then $M_n(t) = \vf\pt\tp\vf\phi(t)$, which reduces the testing problem
to the one for simple nulls.

Oftentimes, there is no direct knowledge on $\vf\pt$ or $\pa$ so one
has to set $C=\Delta$; constraints on $\vfc$ are indirectly imposed
through the condition $\vfc\tp\vf\phi\in \CA_{n,t}$.   Then $M_n(t)$
takes the form
\begin{gather} \label{eq:lp-sequential-pval}
  M_n(t) = \sup\{
    \vfc\tp \vf\phi(t): \vfc\in \Delta, \
    \vfc\tp \vf\phi \in \CA_{n,t}
  \}.
\end{gather}
In Section \ref{sec:simul}, we will consider the case where $C$ can be
chosen smaller than $\Delta$, and in Section \ref{sec:mle}, a case
where substantial knowledge on $\vf\pt$ can be attained by estimation
will be considered.

Recall that $(1-\pa)\vf\pt\tp\vf\phi(t)$ is the population fraction of
true nulls with $X_i\in D_t$.  In order for $M_n(t)$ not to
underestimate the fraction, a basic requirement is $M_n(t) \ge
(1-\pa)\vf\pt\tp \vf\phi(t)$.  In general, since $\CA_{n,t}$ is
random, this requires that $\CA_{n,t}$ have the property that as
long as $n$ is large enough, with probability close to 1,
$(1-\pa)\vf\pt\tp \vf\phi\in \CA_{n,t}$ for all $t\in \CI$.

A basic fact to use in order to satisfy the condition is that,
almost surely, as $n\toi$,
\begin{align*}
  \sup_t\Abs{\EP_n(t) \to \PD(t)}\to 0,
\end{align*}
where $\PD(t)$ is the distribution function of $s_i = s(X_i)$ defined in
\eqref{eq:s}, i.e.
\begin{align*}
  \PD(t) = (1-\pa)\vf\pt\tp \vf\phi(t) + \pa \PA(D_t).
\end{align*}
Then, with probability close to 1, $(1-\pa)\vf\pt\tp \vf\phi$ is
less than $\EP_n(t)$ plus a small margin.  Moreover, $\PD(t) -
(1-\pa)\vf\pt\tp \vf\phi(t) = \pa \PA(D_t)$ is increasing in
$t$.  Then for $n\gg 1$, with probability close to 1,
\begin{align*}
  \EP_n(u)-(1-\pa)\vf \pt\tp\vf\phi(u) > \EP_n(v) -
  (1-\pa)\vf \pt\tp \vf\phi(v) - \epsilon_n, \text{ for all } u>v.
\end{align*}
Therefore, in calculating $M_n(t)$, the maximization can be
constrained to those $\vfc$ such that, when they replace
$(1-\pa)\vf\pt$, the inequalities still hold.

\subsection{Construction using data sequentially}
Given the relative ease to establish FDR control by using a stopping
time as the random cut-off for rejection, we first consider a
construction of $\CA_{n,t}$ that allows a stopping time to be
defined.

Incorporating the facts discussed just now, a basic form of
$\CA_{n,t}$ is
\begin{align}
  \CA_{n,t} = \Cbr{\!\!
    \begin{array}{c}
      h\in C(\CI): h(s_i)\le \EP_n(s_i) + \epsilon_n
      \text{ for } s_i\ge t
      \\[.5ex]
      \EP_n(t_2)-\EP_n(t_1) \ge  h(t_2)-h(t_1) - \epsilon_n \text{ for}
      \\[.5ex]
      t_1, t_2 \in \CT_n\text{ with } t\le t_1<t_2
    \end{array}
    \!\!
  },
  \label{eq:constraint-seq}
\end{align}
where $\CT_n\subset\CI$ is a finite set of points.  Although
$\CT_n$ can contain any number of points, to reduce
computation, the number of points in $\CT_n$ needs to be relatively
small.

It is easy to see $M_n(t)=0$ if $t\le\inf\CI$.  Some other useful
properties of $M_n(t)$ are as follows.

\begin{lemma}\label{lemma:Mt-property}
  $M_n$ is always nondecreasing.  Furthermore, if $\phi_i\in
  C(\Reals)$ for all $i$, then almost surely,  1) $M_n$ is continuous
  at every $t$ other than $\eno s n$ and 2) it is left-continuous and
  has a right-hand limit at each $s_i$.
\end{lemma}

The continuous and discrete versions of the BH procedure using
$M_n(t)$ are described below.  Similar to Procedure
\ref{proc:equivalent}, the two versions are equivalent.  As in
Procedure \ref{proc:max}, the random variable $\tau$ in the continuous
version is a stopping time. 

\begin{procedure}\rm \label{proc:finite}
  Given control parameter $\alpha \in (0,1)$, let 
  \begin{align*}
    \CI_R = 
    \Cbr{
      t\in \CI:\ 
      \frac{M_n(t)}{\alpha} \le \frac{R_n(t)\vee 1}{n}
    }.
  \end{align*}
  If  $\CI_R\not=\emptyset$, set $\tau=\sup\CI_R$ and reject $H_i$ if
  and only if $s_i\le \tau$.  Otherwise, set $\tau=\inf\CI$ and accept
  all $H_i$.

  Equivalently, sort $s_i$ into $s\Sb 1 \le \ldots \le s\Sb n$ and set
  $s\Sb 0=\inf\CI$.  Reject $H_i$ if and only if $s_i\le s\Sb R$,
  where
  \begin{align*}
    R = \max\Cbr{
      i\ge 0: \frac{M_n(s\Sb i)}{\alpha} \le
      \frac{R_n(s\Sb i)\vee 1}{n}
    }.
    \eqed
  \end{align*}
\end{procedure}

For each $i$, $M_n(s\Sb i)$ is the maximum of $\vfc\tp\vf\phi(s\Sb
i)$, with $c_k$ satisfying
\begin{itemize}
\item[1)] $c_k\ge 0$, $\sum c_k \le 1$;
\item[2)] $\vfc\tp\vf\phi(s\Sb j) \le \EP_n(s\Sb j) + \epsilon_n$ for
  $j\ge i$;
\item[3)] $\EP_n(t_2) - \EP_n(t_1)\ge \sum_{k=1}^L c_k [\phi_k(t_2) -
  \phi_k(t_1)] + \epsilon_n$ for $t_1, t_2 \in \CT_n$ with $s\Sb i\le
  t_1<t_2$.
\end{itemize}
All the constraints are linear.  As a result, $M_n(s\Sb i)$ can be
computed by linear programming.  The computation is termed sequential
because each $M_n(s\Sb i)$ is computed based on the data greater than
$s\Sb i$.  Therefore, if we imagine that $s\Sb i$ are input one by
one, starting with the largest one, then $M_n(s\Sb i)$ can be computed
only after all $s\Sb j$, $j\ge i$, have been input.

The FDR control of Procedure \ref{proc:finite} is given in the next
result.  The main tool for the proof is martingale stopping time and
the Dvoretzky-Kiefer-Wolfowitz (DKW) inequality \cite{massart:90}. 
\begin{theorem} \label{thm:lp-sequential}
  Suppose 1) $\phi_i\in C(\Reals)$, 2) $\vf\pt\tp\vf\phi(t)>0$ for all
  $t\in \CI$ and 3) $\PA(D_t)$ in continuous in $t$.  Then for $n\ge
  1$, provided $\exp(-2n\epsilon_n^2) \le 1/2$, Procedure
  \ref{proc:finite} satisfies
  \begin{align*}
    \fdr \le \alpha + 2 (1+|\CT_n|) \exp(- 2n\epsilon_n^2) +
    E\Sbr{\frac{\cf{R>0}}{R\vee 1}}.
  \end{align*}
\end{theorem}
The bound contains terms in addition to $\alpha$.  For
appropriate $\epsilon_n$ and $\CT_n$, the term $2(1+|\CT_n|) \exp(-
2n\epsilon_n^2)$ is $o(1)$ as $n\toi$.  Under certain conditions, $R$
is of the same order as $n$ and hence the bound shows $\fdr$ can be
asymptotically controlled at $\alpha$.  However, the simulation study
in Section \ref{sec:simul} indicates that usually the realized $\fdr$
is substantially lower than $\alpha$, which is reasonable because
$M_n(t)$ is an overestimation of $(1-\pa)\vf\pt\tp\vf\phi(t)$.

\subsection{Construction using entire data}
In place of $\CA_{n,t}$ which depends on $t$, we can use a single
family of functions $\CA_n$.  In order to impose maximum amount of
linear constraints, $\CA_n$ should incorporate all $X_i$.  Based on the
same considerations underlying \eqref{eq:constraint-seq}, we define
\begin{align}  \label{eq:constraint-global}
  \CA_n = \Cbr{
    \begin{array}{c}
      h\in C(\CI): h(s_i)\le \EP_n(s_i) + \epsilon_n \text{ for all }
      s_i \\[.5ex]
      \EP_n(t_2)-\EP_n(t_1) \ge  h(t_2)-h(t_1) - \epsilon_n
      \\[.5ex]
      \text{for } t_1, t_2\in \CT_n \text{ with } t_1<t_2
    \end{array}
  }.
\end{align}
Corresponding to \eqref{eq:lp-sequential-pval}, for $t\in\CI$, define
\begin{align} \label{eq:lp-global-pval}
  M_n(t) = \sup\Cbr{
    \vfc\tp \vf\phi(t):\, \vfc\in \Delta, \,
    \vfc\tp \vf\phi \in \CA_n
  }.
\end{align}
It is easy to see that $M_n$ is nondecreasing.  Therefore,
corresponding to Procedure \ref{proc:finite}, the following BH
procedure obtains.

\begin{procedure}\rm \label{proc:finite2}
  Given control parameter $\alpha \in (0,1)$, let
  \begin{align*}
    \CI_R = 
    \sup\Cbr{
      t\in \CI:\ 
      \frac{M_n(t)}{\alpha} \le \frac{R_n(t)\vee 1}{n}
    }.  
  \end{align*}
  If $\CI_R\not=\emptyset$, set $\tau=\sup\CI_R$ and reject
  $H_i$ if and only if $s_i\le \tau$.  Otherwise, set $\tau=\inf\CI$
  accept all $H_i$.

  Equivalently, sort $s_i$ into $s\Sb 1 \le \ldots \le s\Sb n$ and set
  $s\Sb 0=\inf\CI$.  Reject $H_i$ if and only if $s_i\le s\Sb R$,
  where
  \begin{align*}
    R = \max\Cbr{
      i\ge 0: \frac{M_n(s\Sb i)}{\alpha} \le \frac{R_n(s\Sb i)\vee 1}{n}
    }.
    \eqed
  \end{align*}
\end{procedure}

Like Procedure \ref{proc:finite}, $M_n(s\Sb i)$ can be computed using
linear programming.  For comparision, we list the constraints for the
maximization.
For each $i$, $M_n(s\Sb i)$ is the maximum of $\vfc\tp\vf\phi(s\Sb
i)$, with $c_k$ satisfying
\begin{itemize}
\item[1)] $c_k\ge 0$, $\sum c_k \le 1$;
\item[2)] $\vfc\tp\vf\phi(s\Sb j) \le \EP_n(s\Sb j) + \epsilon_n$ for
  all $j=1,\ldots,n$.
\item[3)] $\EP_n(t_2) - \EP_n(t_1)\ge \sum_{k=1}^L c_k [\phi_k(t_2) -
  \phi_k(t_1)] + \epsilon_n$ for all $t_1, t_2 \in \CT_n$ with
  $t_1<t_2$.
\end{itemize}
It is worth pointing out that although the set of constraints on
$\vfc$ is the same for all $s_i$, for different $i$, because
$\vf\phi(s_i)$ are different, the value of $\vfc$ that yields
$M_n(s_i)$ will be different.

Unlike Procedures \ref{proc:max} and \ref{proc:finite}, since $\tau$
in Procedure \ref{proc:finite2} is determined by the entire $\eno s
n$, it is not a stopping time.  Because the martingale stopping time
argument cannot be used to establish FDR control for finite $n$, we
will work out an asymptotic statement instead.

For $s\in \Reals$ and $S\subset\Reals$, denote the distance from $s$
to $S$ by $d(s,S) = \inf\{|s-t|:\,t\in S\}$.  Define $\delta(S,T) =
\sup\{d(s,S):\, s\in T\}$ for $S$, $T\subset \Reals$.  A sequence
$S_n$ of finite sets is said to be increasingly dense in $T$ if for
any $r>0$, $\delta(S_n,\, T\cap [-r,r])\to 0$ as $n\toi$.
\begin{theorem} \label{thm:FDRg}
  Suppose 1) all $\phi_i$ are continuous and $\vfc\tp\vf\phi$ is
  strictly increasing in $\CI$ 2) $\PA(D_t)$ is continuous in $t$, and
  3) as $n\toi$, $\epsilon_n\to 0$, $n\epsilon_n^2\toi$ and $\CT_n$ is
  increasingly dense in $\CI$.  Then, under Assumption A given below,
  for Procedure \ref{proc:finite2}, $\Lsup_{n\toi} \fdr \le \alpha$.

  Furthermore, asymptotically the procedure is equivalent to the one
  that reject $H_i$ if and only if $s_i\le t_*$, where $t_*$ is
  defined in \eqref{eq:tau-fix} below.
\end{theorem}

Intuitively, as $n\toi$, in certain sense $\CA_n$ should tend to $\CA
= \{h \in C(\CI)$: $\PD - h\ge 0$ is nondecreasing$\}$.  Consequently,
$M_n(t)$ should tend to
\begin{align} \label{eq:mt}
  m(t) =
  \sup\{
    \vfc\tp \vf\phi(t):\, \vfc\in\Delta, \, \vfc\tp \vf\phi\in
    \CA
  \}.
\end{align}
If this is true, then, as in \cite{genovese:was:02}, the asymptotic of
$\fdr$ as $n\toi$ may be characterize by a fixed point derived from
$m(t)$ and $\PD(t)$.  Let
\begin{align}
  t_* = \sup\Cbr{t\in\CI:\ m(t) \le \alpha \PD(t)}.
  \label{eq:tau-fix}
\end{align}
\subsection*{\bf Assumption A} $t_*\in \CI$ and there is $t_0 <t_*$,
such that $m(t) < \alpha \PD(t)$ on $(t_0, t_*)$.

\section{Numerical study} \label{sec:simul}

\subsection{Setup}
Because the properties of $M_n(t)$ in \eqref{eq:lp-sequential-pval}
and \eqref{eq:lp-global-pval} are hard to keep track of, it is
difficult to analyze the power and pFDR of Procedures
\ref{proc:finite} and \ref{proc:finite2}.  We resort to 
numerical simulations to get a handle to these two quantities.  For
comparison, Procedure \ref{proc:max} and the BH procedure with the
prior probabilities $\eno\pt L$ being known are also included.

We only consider univariate observations.  To use lower-tail
$p$-values, we set $D_t = (-\infty, t]$.  By \eqref{eq:s}, if an
observation $X$ takes value $x$, then $s(X) = x$ and hence
$\phi_k(s(X)) = \PT_k(x)$, the left-tail $p$-value of $X$ under
$\PT_k$.  Also, given observations $\eno X n$, from $R_n(t) =
\sum_{i=1}^n \cf{X_i\in D_t}$, $R_n(X_i)$ is the rank of $X_i$.

In each simulation, we draw iid samples $\eno X n$ from a
mixture distribution $(1-\pa)\sum_{k=1}^L \pt_k \PT_k(x)+\pa \PA(x)$,
where $G\not=\eno F L$.  To test nulls
\begin{align*}
  H_i: X_i \sim \PT_k \text{ for some } k, \quad i=1,\ldots,n,
\end{align*}
we compute four types of $p$-values:
\begin{itemize}
\item[1)] $\pseq i = M_n(X_i)$ defined by
  \eqref{eq:lp-sequential-pval} and \eqref{eq:constraint-seq}, where
  ``seq'' in the subscript stands for  ``sequential'', indicating that
  as the calculation of $M_n(X_i)$ precedes to smaller $X_i$, linear
  constraints are added sequentially;
\item[2)] $\pglb i = M_n(X_i)$ defined by \eqref{eq:constraint-global}
  and \eqref{eq:lp-global-pval}, where ``glb'' in the subscript stands
  for ``global'', indicating that $M_n(X_i)$ are calculated under
  linear constraints imposed by all $\eno X n$;
\item[3)] $\pmax i = \max_k\PT_k(X_i)$;
\item[4)] $\pmix i = \sum_k \pt_k \PT_k(X_i)$, i.e., the $p$-value of
  $X_i$ when the values of $\pt_1$, \ldots $\pt_L$ are known.
\end{itemize}

The computation of $\pseq i$ and $\pglb i$ is done by linear
programming.  By \eqref{eq:lp-sequential-pval} and
\eqref{eq:lp-global-pval}, both are maxima of $\vfc\tp\vf\PT(X_i) =
c_1 \PT_1(X_i)+\cdots+c_L\PT_L(X_i)$.  In the simulations, the
constraints are a little different from those basic ones given in
\eqref{eq:constraint-seq} and \eqref{eq:constraint-global}.  However,
the analysis is the same.

Denote by $\bar\Gamma^*(z; \alpha, \beta)$ the $z$-th upper-tail
quantile of the Gamma distribution with shape parameter $\alpha$ and
scale parameter $\beta$.  For $i=1,\ldots,n$, to compute $\pseq i$,
the constraints on $\eno c L$ are
\begin{itemize}
\item[1)] $c_k\ge 0$, $\sum c_k \le 1$;
\item[2)] $\vfc\tp\vf\PT(X_j)\le u(X_j)$ for $X_j\ge X_i$, where
  \begin{align*}
    u(X_j)
    =
    \begin{cases}
      \displaystyle
      \nth n
      \bar\Gamma^*\Grp{\nth n; R_n(X_i), \nth{0.95}},
      & \text{if } R_n(X_j) \le n^{0.2}, \\[.5ex]
      \EP_n(X_j) + \epsilon_n & \text{otherwise};
    \end{cases}
  \end{align*}
\item[3)] $\EP_n(t_2) - \EP_n(t_1)\ge \vfc\tp [\vf\PT(t_2) -
  \vf\PT(t_1)] + \epsilon_n$ for $t_1, t_2 \in \CT_n$ with $X_i\le
  t_1<t_2$, where $\EP_n(t) = R_n(t)/n$.
\end{itemize}
In all the simulations, $\epsilon_n = \sqrt{\ln n/n}$ and $\CT_n$
consists of $\lfloor (\ln n)^2\rfloor$ equally spaced points with the
first and last ones being $\min X_i$ and $\max X_i$.

The only difference between the above constraints and those in
\eqref{eq:constraint-seq} is the modified upper bound $u(X_j)$ when
$R_n(X_j)\le n^{0.2}$.  This aims to impose stronger constraint on
$c_k$.  In the definition of $u(X_j)$, $n^{0.2}$ can be changed to any
$a_n=o(n)$ and the scale parameter $1/0.95$ to any $1/\beta$ with
$\beta\in (0,1)$.  As Appendix \ref{sec:append-numeric} shows, at
control parameter $\alpha$, Procedure \ref{proc:finite} using $\pseq
i$ computed under the above constraints obtains
\begin{align} \label{eq:lp-seq-modify}
  \fdr \le \alpha + r_n + E\Sbr{\frac{\cf{R>0}}{R\vee 1}}.
\end{align}
with $r_n \to 0$ as $n\toi$.

With similar modifications to \eqref{eq:constraint-global}, for
$i=1,\ldots, n$, to compute $\pglb i$, the constrains on $\eno c L$
are
\begin{itemize}
\item[1)] $c_k\ge 0$, $\sum c_k \le 1$;
\item[2)] $\vfc\tp\vf\PT(X_j)\le u(X_j)$ for all $X_j\ge X_i$; and 
\item[3)] $\EP_n(t_2) - \EP_n(t_1)\ge \vfc\tp [\vf\PT(t_2) -
  \vf\PT(t_1)] + \epsilon_n$ for all $t_1, t_2 \in \CT_n$ with
  $t_1<t_2$.
\end{itemize}

We then apply the BH procedure to the above $p$-values, specifically,
Procedure \ref{proc:finite} to $\pseq i$, Procedure \ref{proc:finite2}
to $\pglb i$, Procedure \ref{proc:equivalent} to $\pmax i$, and the BH
procedure to $\pmix i$.  For each set of $\eno\PT L$ and $\PA$, we
draw $1000$ iid samples of $\eno X n$ with $n=5000$.  In this case,
$r_n \le 9.7\NE{-3}$ in \eqref{eq:lp-seq-modify}; see Appendix 
\ref{sec:append-numeric}.  The power, FDR and pFDR of each procedure
are calculated by averaging over the samples.  Throughout, $\pa =
0.05$.

All the simulations are conducted in R language \cite{R}; $\pseq i$
and $\pglb i$ are computed by the R linear programming package
\texttt{glpk}.

\subsection{Results}
We conduct 5 groups of simulations.  The parameters of the simulations
are shown in Table \ref{table:distributions}.

\begin{table}[ht]
  \renewcommand{\arraystretch}{1.2}
  \begin{center}
    \begin{tabular}{|c|c|c|c|}
      \hline
      & $\eno\PT L$ & $\eno\pt L$ & $\PA$ \\\hline
      1
      & $N(0,1)$, $N(-1,1)$,  $N(-2,1)$
      & .75, .15, .1 
      & $N(-4,1)$\\\hline
      2
      &  $t_{20}$, $t_{20,-1}$, $t_{20,-2}$
      & .75, .15, .1
      & $t_{20,-4}$ \\\hline
      3
      & $N(0,1)$, $N(-1,1)$, $N(-2,1)$
      & .6, .25, .15
      & $N(-4,1)$ \\\hline
      4
      & $N(0, 1)$, $N(-1,1.5)$, $N(-2,1.5)$
      & .75, .15, .1 
      & $N(-4,1)$ \\\hline
      5
      & $N(\mu,1)$, $\mu=0,-1,-2,-3,-4$
      & .65, .15, .1, .05, .05
      & $N(-5,1)$ \\\hline
    \end{tabular}
  \end{center}
  \caption{\rm Parameters for the simulations.  $\PT_k$ are null
    distributions, $\pt_k$ their prior probabilities, and $\PA$ the
    distribution under false nulls.  In each simulation, $\pa=0.05$.
    $t_{n,c}$ denotes the noncentral $t$ distribution with $n$ df and
    noncentrality $c$.
    \label{table:distributions}
  }
\end{table}

The results of the simulations are summarized in Table
\ref{table:FDR-power}.  In all the simulations, the control parameter
$\alpha$ is equal to 0.25.  As expected, because $\pmix i$
incorporate $\eno \pt L$, which is information not accessible by the
other types of $p$-values, they yield the highest power with
substantial margin. On the other hand, $\pseq i$ and $\pglb i$ yield
substantially higher power than $\pmax i$.  This shows that even when
$\eno\pt L$ are unknown, by utilizing properties of empirical
processes to reduce overestimation of $p$-values, the power of the BH
procedure can still be significantly improved.

\begin{table}[ht]
  \renewcommand{\arraystretch}{1.2}
  \begin{center}
    \begin{tabular*}{.94\textwidth}{@{\extracolsep\fill}|c|ccc|ccc|}
      \hline
      simul.
      & \multicolumn{3}{c|}{\makebox[.38\textwidth]{1}}
      & \multicolumn{3}{c|}{\makebox[.38\textwidth]{2}}
      \\
      & power & FDR & pFDR & power & FDR & pFDR \\\hline
      $\pseq i$
      & .495 & 8.61$\NE{-2}$ & 8.61$\NE{-2}$
      & .236 & 8.57$\NE{-2}$ & 8.57$\NE{-2}$
      \\
      $\pglb i$
      & .494 & 8.60$\NE{-2}$ & 8.60$\NE{-2}$
      & .235 & 8.57$\NE{-2}$ & 8.57$\NE{-2}$
      \\
      $\pmax i$
      & .223   & 2.55$\NE{-2}$ & 2.55$\NE{-2}$ 
      & .035 & 2.87$\NE{-2}$ & 3.18$\NE{-2}$
      \\
      $\pmix i$
      & .770 & .238 & .238
      & .634 & .238 & .238
      \\\hline
    \end{tabular*}
    \\
    \begin{tabular*}{.94\textwidth}{@{\extracolsep\fill}|c|ccc|ccc|}
      \hline
      simul.
      & \multicolumn{3}{c|}{\makebox[.38\textwidth]{3}}
      & \multicolumn{3}{c|}{\makebox[.38\textwidth]{4}}
      \\
      & power & FDR & pFDR & power & FDR & pFDR \\\hline
      $\pseq i$
      & .449 & .103 & .103
      & 4.82$\NE{-4}$ & 6.95$\NE{-2}$ & .465
      \\
      $\pglb i$
      & .449 & .102 & .102
      & 4.82$\NE{-4}$ & 6.95$\NE{-2}$ & .465
      \\
      $\pmax i$
      & .229 & 3.77$\NE{-2}$ & 3.77$\NE{-2}$ 
      & 8.42$\NE{-5}$ & 1.88$\NE{-2}$ & .523 
      \\
      $\pmix i$
      & .685 & .236 & .236
      & .144 & .226 & .259\\\hline
    \end{tabular*}
    \\
    \begin{tabular*}{.56\textwidth}{@{\extracolsep\fill}|c|ccc|}
      \hline
      simul.
      & \multicolumn{3}{c|}{\makebox[.38\textwidth]{5}}
      \\
      & power & FDR & pFDR \\\hline
      $\pseq i$
      & 4.53$\NE{-2}$ & 6.42$\NE{-2}$ & 6.85$\NE{-2}$
      \\
      $\pglb i$
      & 4.62$\NE{-2}$ & 6.51$\NE{-2}$ & 6.94$\NE{-2}$
      \\
      $\pmax i$
      & 3.22$\NE{-3}$ & 1.68$\NE{-2}$ & 4.00$\NE{-2}$ 
      \\
      $\pmix i$
      & .448 & .239 & .239
      \\\hline
    \end{tabular*}
  \end{center}
  \caption{\rm 
    Performance of the BH procedure applied to different types of
    $p$-values in simulations 1--5.  In each simulation, the  control
    parameter is set at $\alpha = 0.25$.
    \label{table:FDR-power}
  }
\end{table}

In agreement with known results \cite{benjamini:hoc:95,
  storey:tay:sie:04}, the FDR attained by using $\pmix i$ or $\pmax i$
is close to or lower than $(1-\pa)\alpha = 0.2375$.  However, the
large gap between the FDR by using $\pmax i$ and $(1-\pa)\alpha$
indicates that testing based on $\pmax i$ can be very
conservative.  On the other hand, in all the simulations, the FDR
attained by using $\pseq i$ or $\pglb i$ lies between the above two,
substantially lower than the first one but substantially higher than
the second.  Together with the simulation result on power, this shows
that multiple testing based on $\pseq i$ and $\pglb i$ is more
conservative than based on $\pmix i$, but can be much less
conservative than based on $\pmax i$.

The conservativeness of multiple testing based on the $p$-values other
than $\pmix i$ does not necessarily help the control of pFDR.  In
simulations 1 and 3, for each type of $p$-value, the power is
relatively high, implying $P(R\ge 1)\approx 1$.  As a result, the pFDR
is almost identical to the FDR.  In simulations 2, 4 and 5, the power
yielded by $\pmax i$ is low ($\le .05$), and, consistent with this,
the pFDR is substantially higher than the FDR.  In contrast, in
simulations 2 and 5, by using $\pseq i$ or $\pglb i$, the pFDR and FDR
are similar to each other.  The worst case is simulation 4, where the
pFDR is almost twice as high as the control parameter $\alpha = .25$
when $\pseq i$ or $\pglb i$ are used.  Observe that in simulation 4,
negative observations with large absolute values are more likely to be
associated with true nulls than with false nulls.  This explains the
poor control of the pFDR by the BH procedure using $\pseq i$
or $\pglb i$.

\begin{figure}[t]
  \renewcommand{\arraystretch}{0}
  \begin{center}
    \begin{tabular*}{.98\textwidth}{@{\extracolsep\fill} cc}
      Simulation 1 & Simulation 2 \\[2ex]
      \epsfig{file = 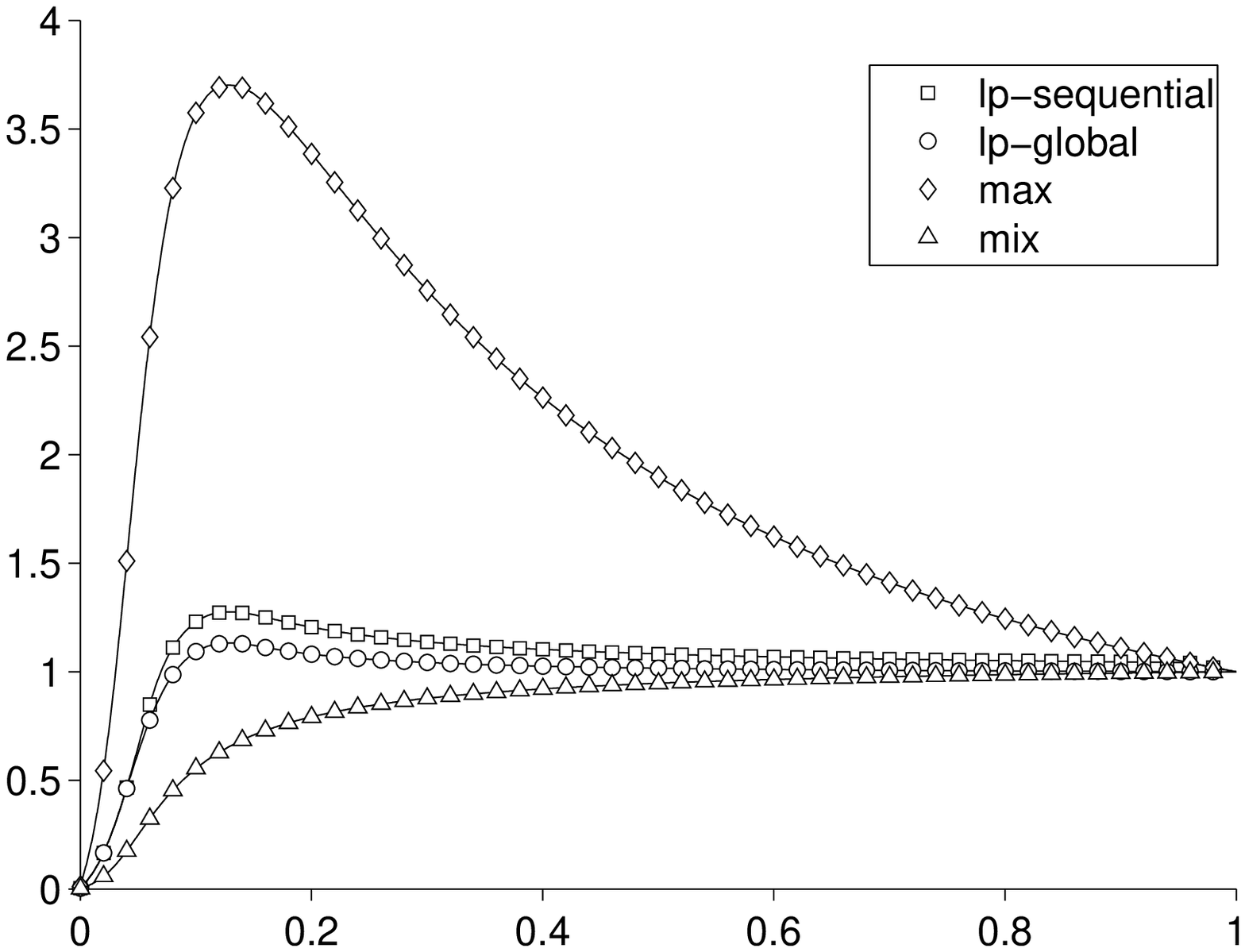,
        width=5.8cm, height=4.7cm} &
      \epsfig{file = 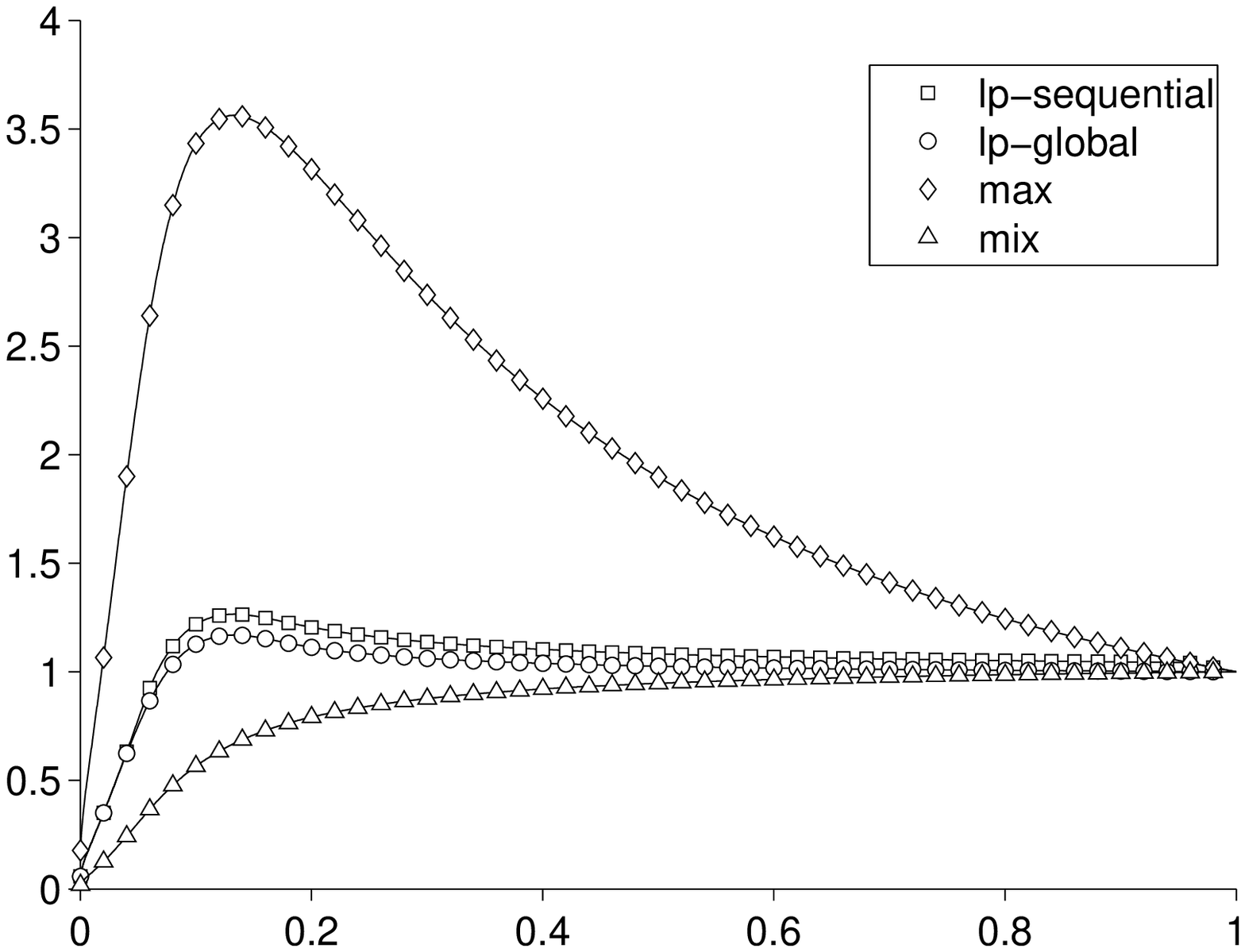,
        width=5.8cm, height=4.7cm} \\[5ex]
      Simulation 3 & Simulation 4 \\[2ex]
      \epsfig{file = 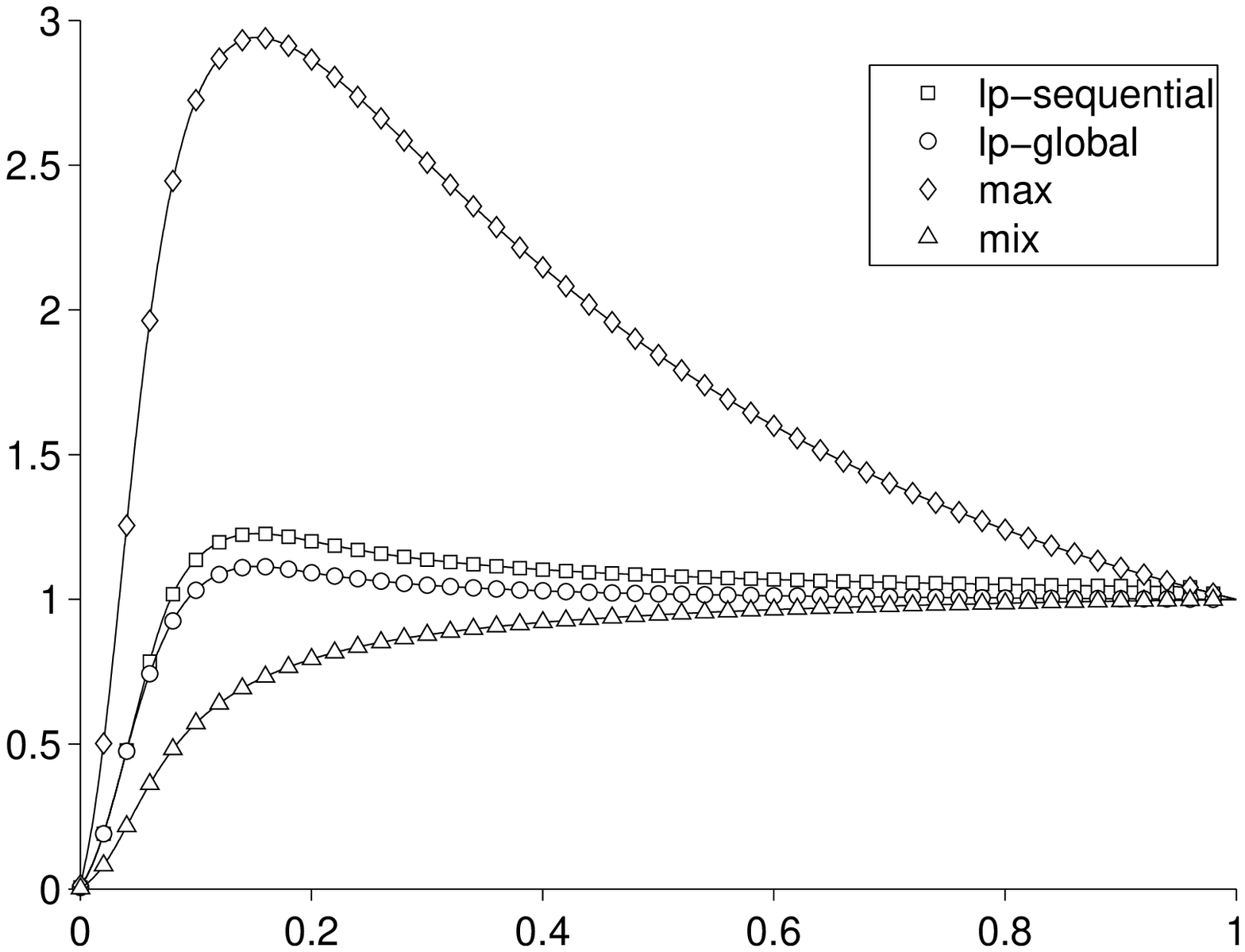,
        width=5.8cm, height=4.7cm} &
      \epsfig{file = 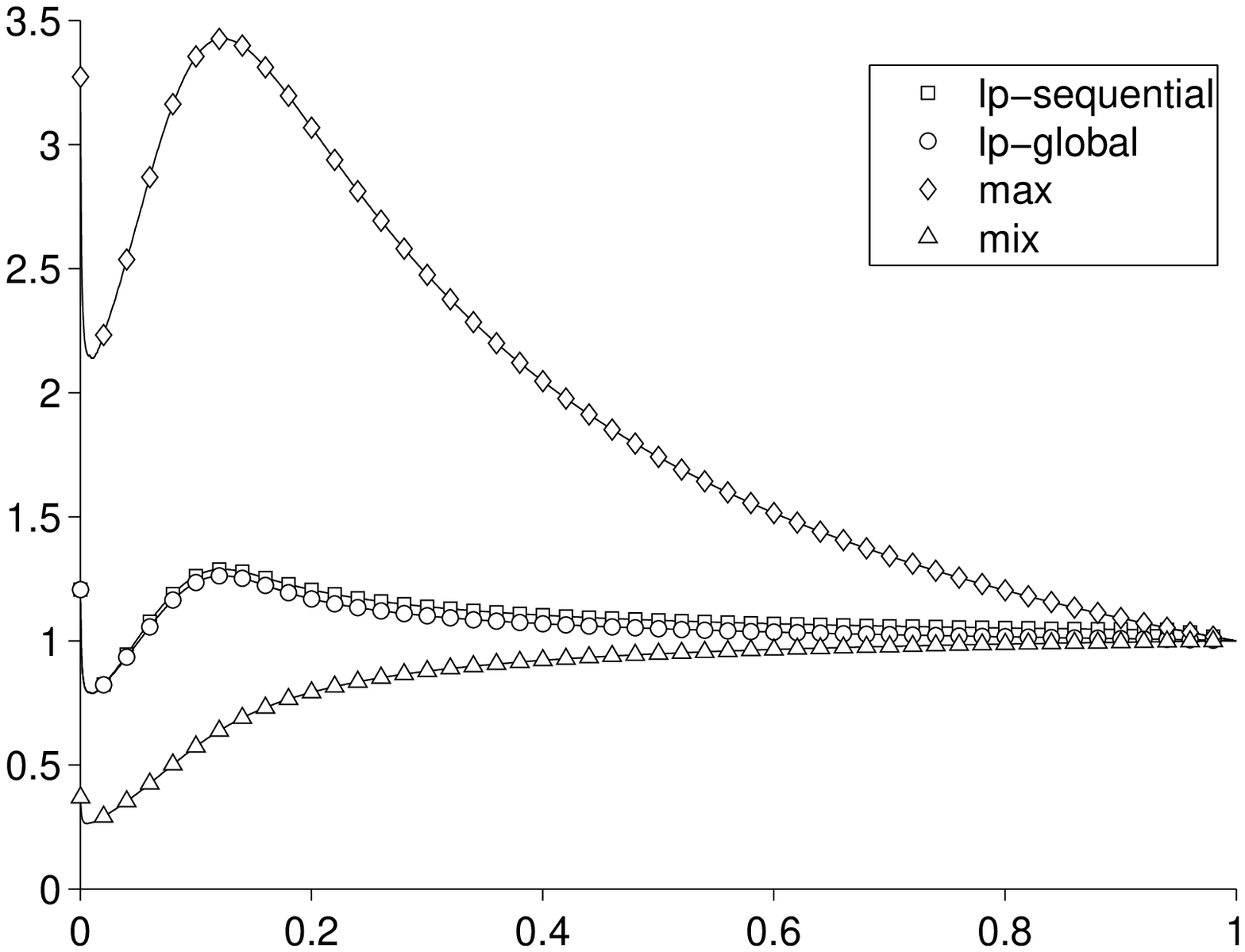,
        width=5.8cm, height=4.7cm} 
    \end{tabular*} \\[5ex]
    Simulation 5 \\[2ex]
    \epsfig{file = 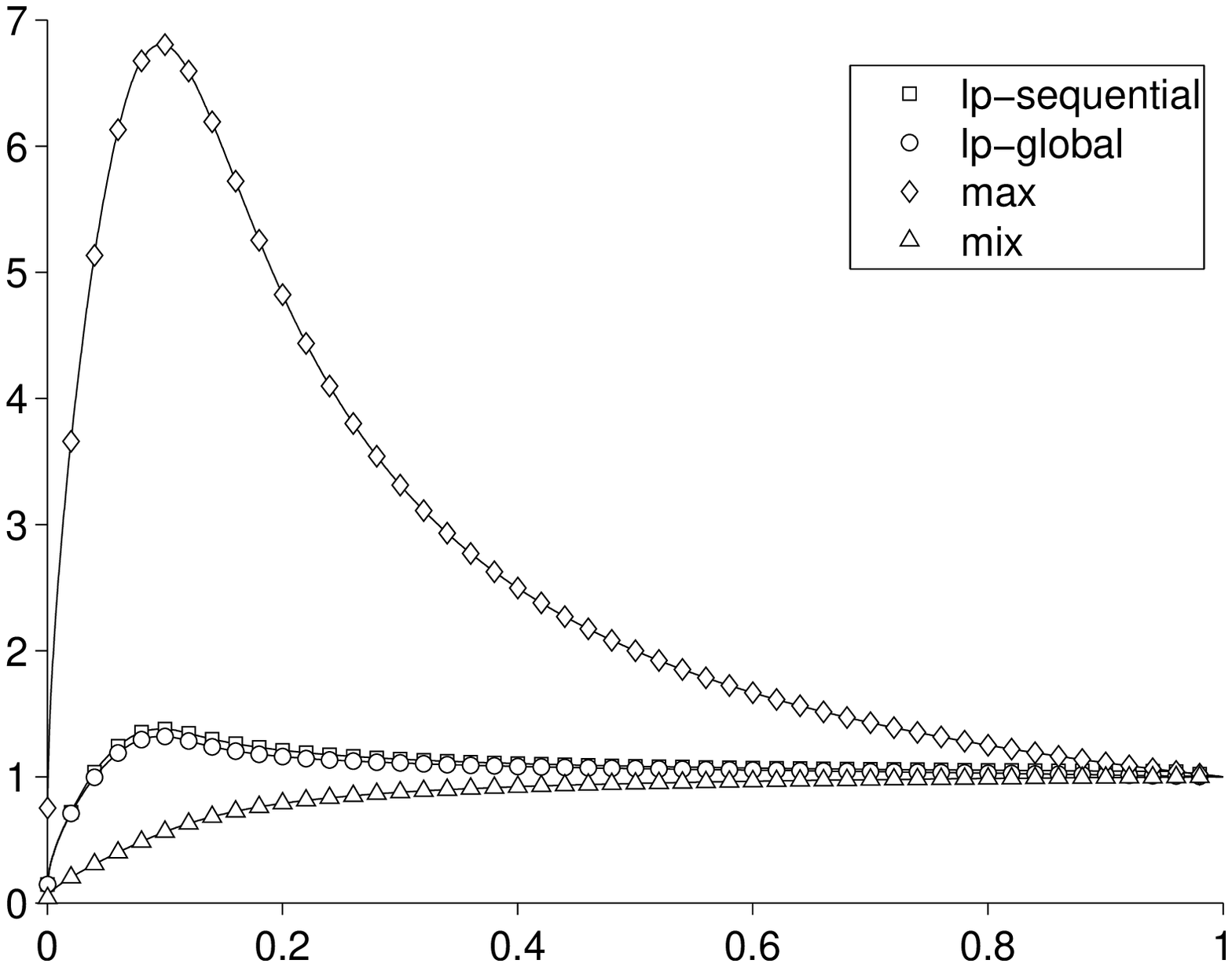,
      width=5.8cm, height=4.7cm} 
  \end{center}
  \caption{\rm\label{fig:pvals}
    Plots of $n\bar p\Sb i/i$ \emph{versus\/} $i/n$ in simulations 1--5
    for different types of $p$-values: $\pseq i$ (``lp-sequential''),
    $\pglb i$ (``lp-global''), $\pmax i$ (``max''), and $\pmix i$ 
    (``mix'').
  }
\end{figure}

To see in more detail why $\pseq i$ and $\pglb i$ in general yield
better multiple testing results than $\pmax i$, we compare the plots
of the $p$-values.  Because all the procedures in the study are
variants of the BH procedure, it is more informative to compare the
plots of $p\Sb{i}/(i/n) = n p\Sb{i}/i$ than to compare those of
$p\Sb{i}$, $i=1,\ldots,n$, where $p\Sb{i}$ is the $i$th smallest
$p$-value of a given type.  Figures \ref{fig:pvals} display the plots
of $n\bar p\Sb{i}/i$ \emph{versus\/} $i/n$ in the simulations, where
$\bar p\Sb{i}$ is the average over the repetitions.  The figure
clearly shows that for small $i/n$, $n\pseq{(i)}/i$ and
$n\pglb{(i)}/i$ are similar to each other, both are substantially
lower than $n\pmax{(i)}/i$, and both increase more rapidly than
$n\pmix{(i)}/i$.  This is consistent with the observation that
multiple testing using $\pseq i$ and that using $\pglb i$ perform
similarly in terms of power, FDR and pFDR, and in general both have
higher power than multiple testing using $\pmax i$ at the same value
of $\alpha$.

\begin{figure}[ht]
  \renewcommand{\arraystretch}{0}
  \begin{center}
    \begin{tabular*}{.98\textwidth}{@{\extracolsep\fill} cc}
      \multicolumn{2}{c}{Simulation 1} \\[2ex]
      \epsfig{file = 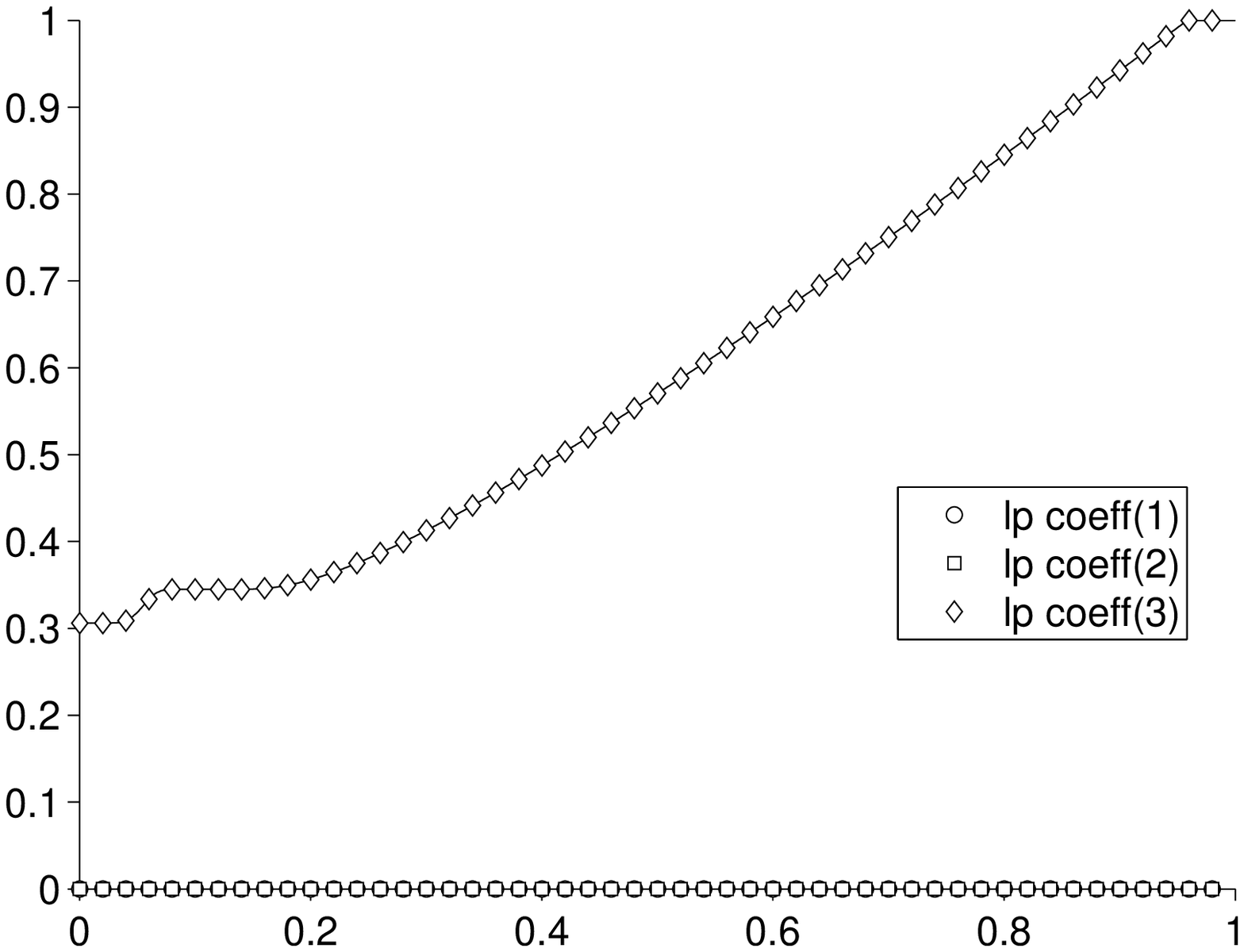,
        width=5.8cm, height=4.7cm} &
      \epsfig{file = 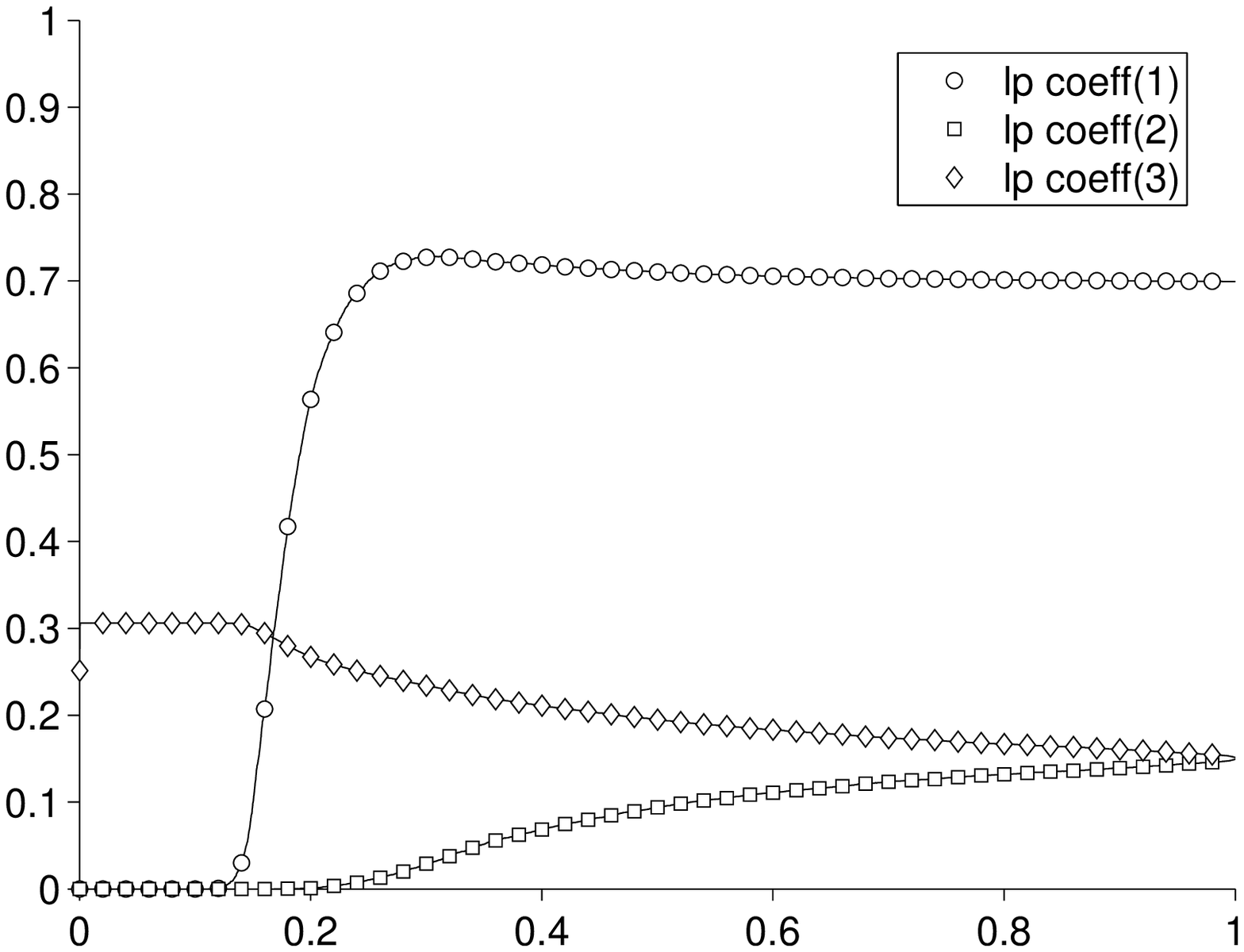,
        width=5.8cm, height=4.7cm} \\[5ex]
      \multicolumn{2}{c}{Simulation 5} \\[2ex]
      \epsfig{file = 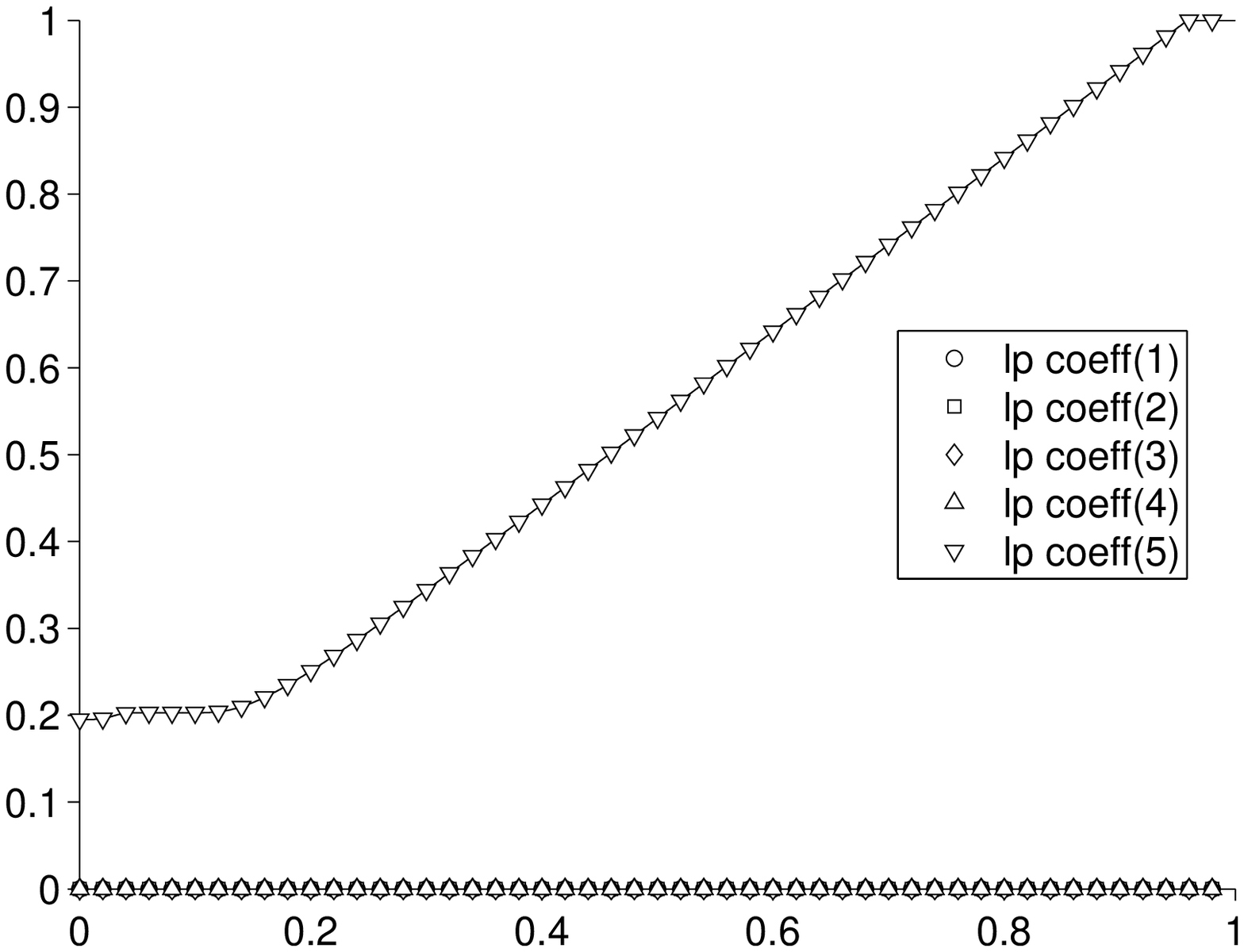,
        width=5.8cm, height=4.7cm} &
      \epsfig{file = 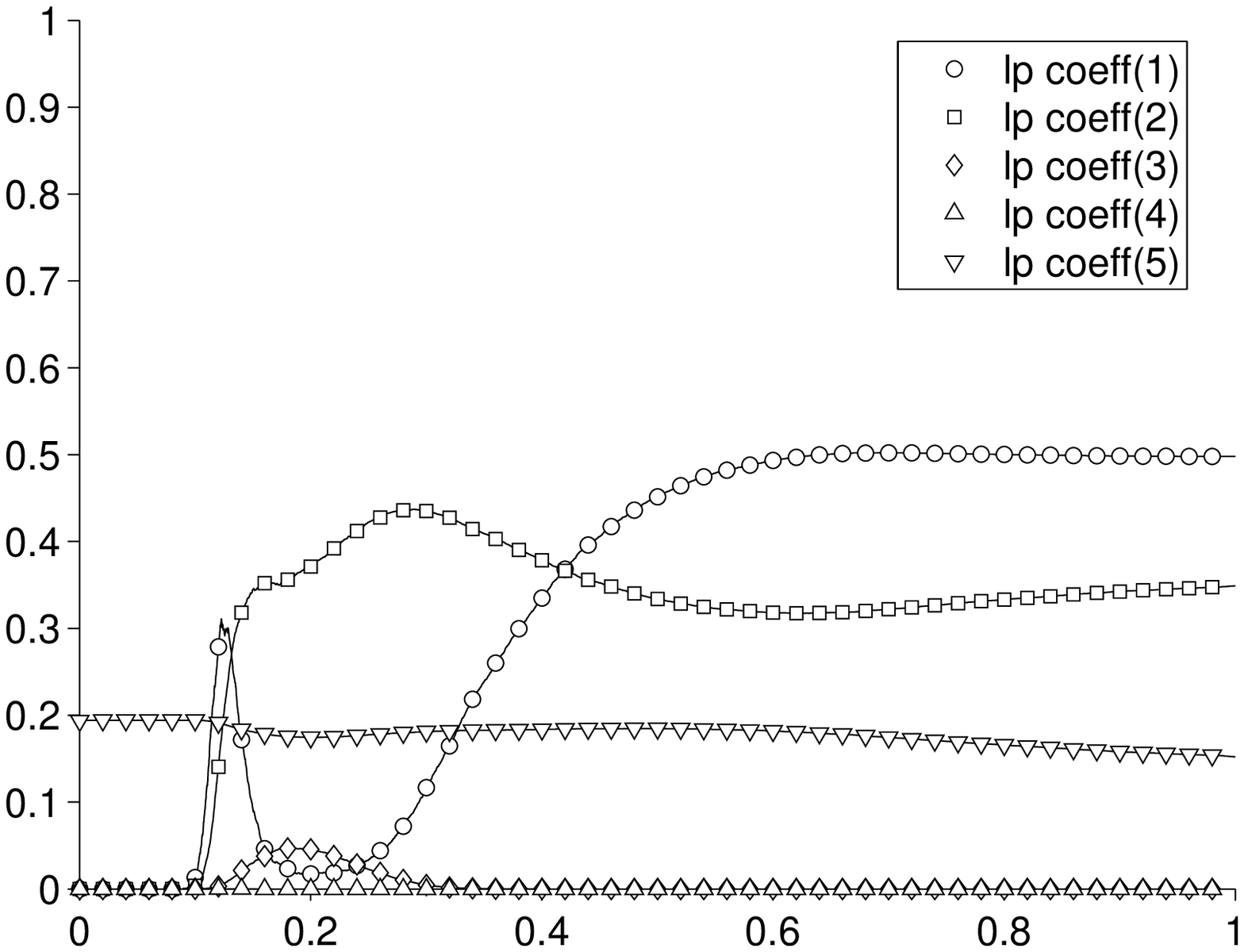,
        width=5.8cm, height=4.7cm} 
    \end{tabular*}
  \end{center}
  \caption{\rm\label{fig:coeff}
    Plots of $c_{k,(i)}$ \emph{versus\/} $i/n$, $k=1,\ldots, n$ in
    simulations 1 and 5, where $c_{1,(i)}, \ldots, c_{L,(i)}$ are the
    coefficients to attain $\pseq{(i)}$ (left) or $\pglb{(i)}$ (right).
  }
\end{figure}

We next look at how $\pseq i$ and $\pglb i$ are computed by linear
programming.  For each $\pseq i$ or $\pglb i$, denote by $c_{1,i},
\ldots, c_{L,i}$ the values of coefficients that yield $p$-values
under the corresponding constraints.  After the $p$-values are sorted,
let $c_{k,(i)}$ be the values corresponding to $\pseq{(i)}$ or
$\pglb{(i)}$.  We plot $c_{k,(i)}$ \emph{versus\/} $i/n$ for
$k=1,\ldots, L$.  Figure \ref{fig:coeff} shows the plots for
simulations 1 and 5.  The plots for the other simulations are
qualitatively similar.  As can be seen, although $\pseq i$ and $\pglb
i$ in the simulations are similar, this is not the case for the
corresponding coefficients $c_{k,i}$.  For each $k$, when $i/n$ is
small, $c_{k,(i)}$ for the two types of $p$-values are similar.
However, as $i/n$ increases, to compute $\pseq i$, essentially only
one $c_k$ stays nonzero.  In all the simulations, this unique $c_k$ is
associated with the last null distribution of the null, i.e., $\PT_L$,
which also has the smallest sup-norm distance from $\PA$ among all
$\PT_k$.  In contrast, to compute $\pglb i$, more complicated
combinations of $\eno c L$ are picked.  This difference between the
coefficients for $\pseq i$ and $\pglb i$ may be partially attributed
to how linear programming is implemented by the package used.
However, it also indicates linear programming may not yield consistent
estimation of $\eno c L$.

Note that in Figure \ref{fig:coeff}, for small $i/n$, the sum of
$c_{k,(i)}$ is quite smaller than 0.4.  Since $\pa=1-\sum c_k$, this
would imply the fraction of false nulls could be as high as 0.6, which
is improbable in many cases.  This raises the possibility that, by
imposing some constraint on the sum of $c_k$, the power may be
improved.  Recall that $\pa=0.05$ in the simulation study.  We
simulate the scenario where it is known that $\pa\le 0.1$.  For both
$\pseq i$ and $\pglb i$, the first constraint on $\eno c L$ is
expanded to become
\begin{itemize}
\item[1')] $c_k\ge 0$, $0.9\le \sum c_k \le 1$.
\end{itemize}

\begin{table}[ht]
  \renewcommand{\arraystretch}{1.2}
  \begin{center}
    \begin{tabular*}{.94\textwidth}{@{\extracolsep\fill}|c|ccc|ccc|}
      \hline
      simul.
      & \multicolumn{3}{c|}{\makebox[.38\textwidth]{1}}
      & \multicolumn{3}{c|}{\makebox[.38\textwidth]{2}}
      \\
      & power & SD$(\frac{R-V}{n-N})$ & pFDR
      & power & SD$(\frac{R-V}{n-N})$ & pFDR \\\hline
      $\pseq i$
      & .495 & 5.20$\NE{-2}$ & 8.61$\NE{-2}$
      & .236 & 6.42$\NE{-2}$ & 8.57$\NE{-2}$
      \\
      $\pglb i$
      & .494 & 5.20$\NE{-2}$ & 8.60$\NE{-2}$
      & .235 & 6.42$\NE{-2}$ & 8.57$\NE{-2}$
      \\
      $\pseq i'$
      & .541 & 5.25$\NE{-2}$ & .103
      & .296 & 6.83$\NE{-2}$ &  9.94$\NE{-2}$
      \\
      $\pglb i'$
      & .541 & 5.25$\NE{-2}$ & .103
      & .296 & 6.83$\NE{-2}$ &  9.94$\NE{-2}$
      \\\hline
    \end{tabular*}
    \\
    \begin{tabular*}{.94\textwidth}{@{\extracolsep\fill}|c|ccc|ccc|}
      \hline
      simul.
      & \multicolumn{3}{c|}{\makebox[.38\textwidth]{3}}
      & \multicolumn{3}{c|}{\makebox[.38\textwidth]{4}}
      \\
      & power & SD$(\frac{R-V}{n-N})$ & pFDR
      & power & SD$(\frac{R-V}{n-N})$ & pFDR \\\hline
      $\pseq i$
      & .449 & 5.35$\NE{-2}$ & .103
      & 4.82$\NE{-4}$ & 1.76$\NE{-3}$ & .465
      \\
      $\pglb i$
      & .449 & 5.35$\NE{-2}$ & .102
      & 4.82$\NE{-4}$ & 1.76$\NE{-3}$ & .465
      \\
      $\pseq i'$
      & .473 & 5.38$\NE{-2}$ & .113
      & 6.26$\NE{-4}$ & 1.86$\NE{-3}$ & .479
      \\
      $\pglb i'$
      & .473 & 5.38$\NE{-2}$ & .113
      & 6.26$\NE{-4}$ & 1.86$\NE{-3}$ & .479 \\\hline
    \end{tabular*}
    \\
    \begin{tabular*}{.56\textwidth}{@{\extracolsep\fill}|c|ccc|}
      \hline
      simul.
      & \multicolumn{3}{c|}{\makebox[.38\textwidth]{5}}
      \\
      & power & SD$(\frac{R-V}{n-N})$ & pFDR\\\hline
      $\pseq i$
      & 4.53$\NE{-2}$ & 3.38$\NE{-2}$ & 6.85$\NE{-2}$
      \\
      $\pglb i$
      & 4.62$\NE{-2}$ & 4.53$\NE{-2}$ & 6.94$\NE{-2}$
      \\
      $\pseq i'$
      & 4.65$\NE{-2}$ & 3.32$\NE{-2}$ & 6.89$\NE{-2}$
      \\
      $\pglb i'$
      & 4.65$\NE{-2}$ & 3.32$\NE{-2}$ & 6.89$\NE{-2}$
      \\\hline
    \end{tabular*}
  \end{center}
  \caption{\rm 
    Performance of the BH procedure applied to $p$-values computed
    under different linear constraints: $\pseq i$ and $\pglb i$ are
    the same as in Table \ref{table:FDR-power}, $\pseq i'$ and
    $\pglb i'$ are computed with the additional constraint $c_1 +
    \cdots + c_L \ge 0.9$.  For each simulation, $\frac{R-V}{n-N}$ is
    the fraction of rejected false nulls among all false nulls in a
    repetition.  The SD is obtained over $1000$ repetitions. 
    \label{table:FDR-power2}
  }
\end{table}

Denote the $p$-values computed with the expanded linear constraints by
$\pseq i'$ and $\pglb i'$, and those computed previously still by
$\pseq i$ and $\pglb i$.  In Table \ref{table:FDR-power2}, the power
and pFDR of the BH procedures when applied to the $p$-values are
compared.  In all the cases, the FDR is substantially lower than
$(1-\pa)\alpha=0.2375$ and hence not shown.  In place of FDR, the
standard deviation of $\frac{R-V}{n-N}$ over 1000 repetitions is
reported.  Recall $R$ is the number of rejections, $V$ that of false
rejections, $n=5000$ is the total number of nulls, and $N$ is number
of true nulls.   In simulations 1--3, there is a small but significant
increase in power by using $\pseq i'$ and $\pglb i'$.  This is not the
case in simulations 4 and 5, where the power is very low for all the 4
types of $p$-values.

\begin{figure}[ht]
  \renewcommand{\arraystretch}{0}
  \begin{center}
    \begin{tabular*}{.98\textwidth}{@{\extracolsep\fill} cc}
      Simulation 1 & Simulation 2 \\[2ex]
      \epsfig{file = 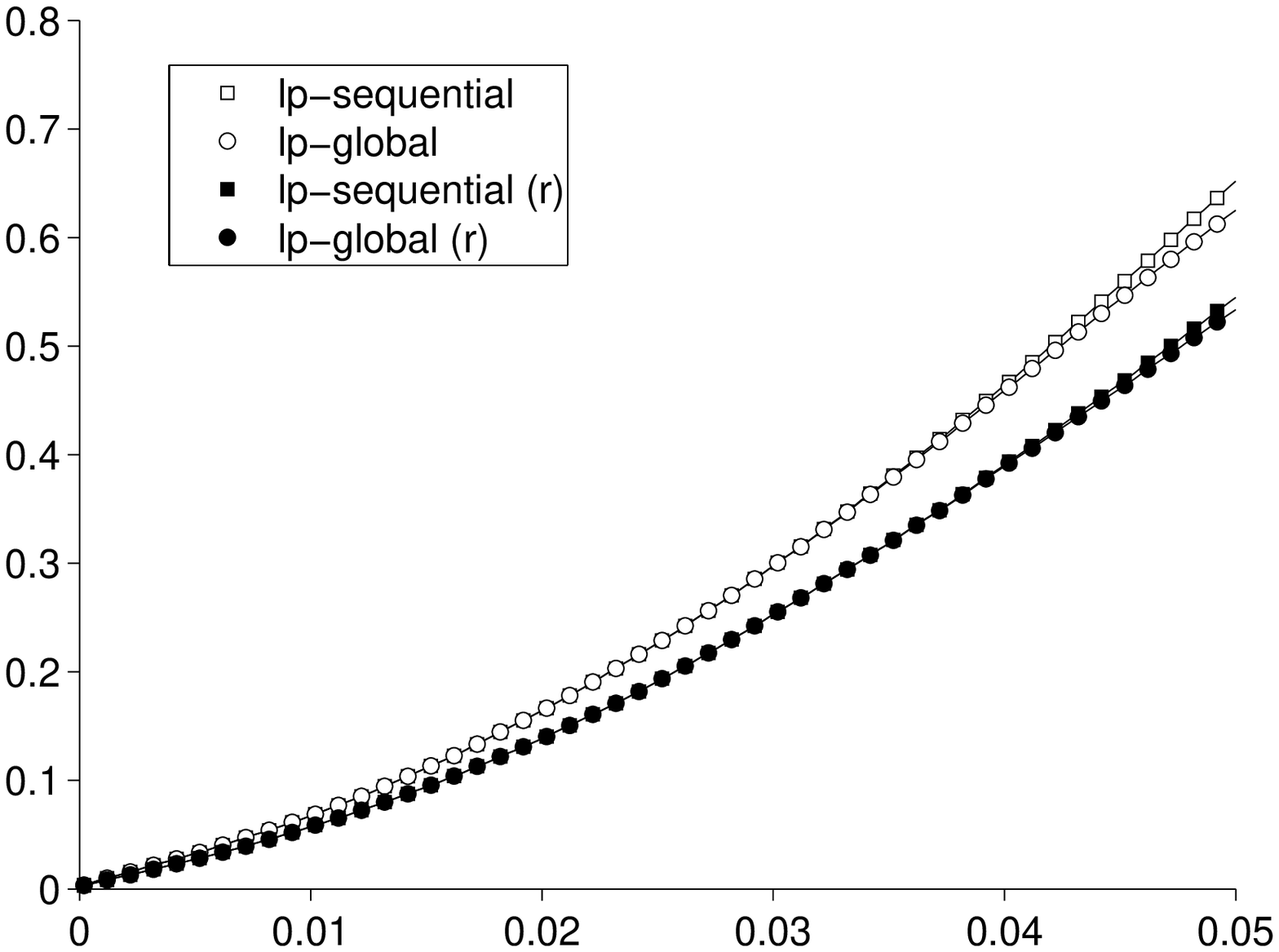, width=5.8cm, height=4.7cm} &
      \epsfig{file = 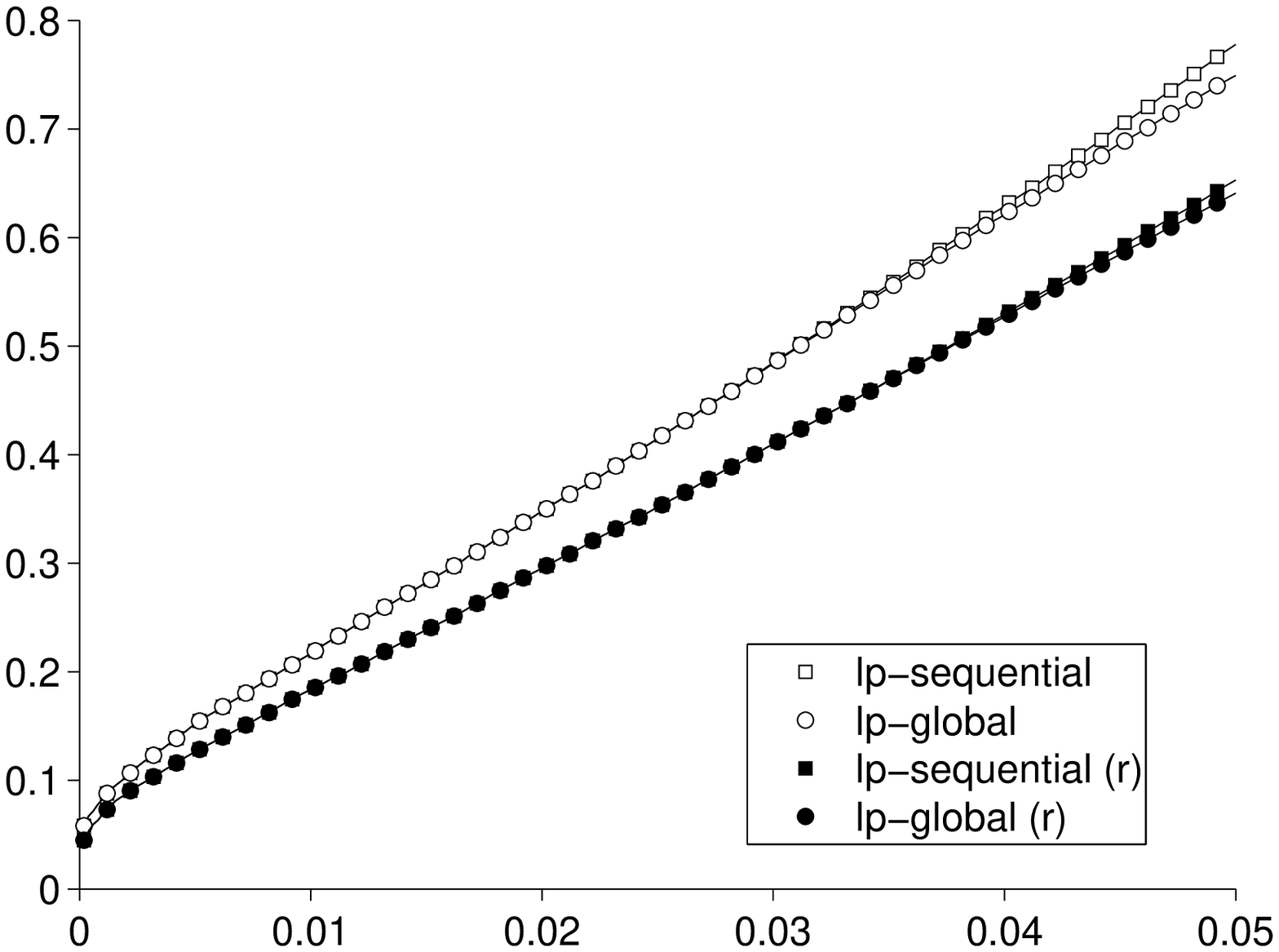, width=5.8cm, height=4.7cm} \\[5ex]
      Simulation 3 & Simulation 4 \\[2ex]
      \epsfig{file = 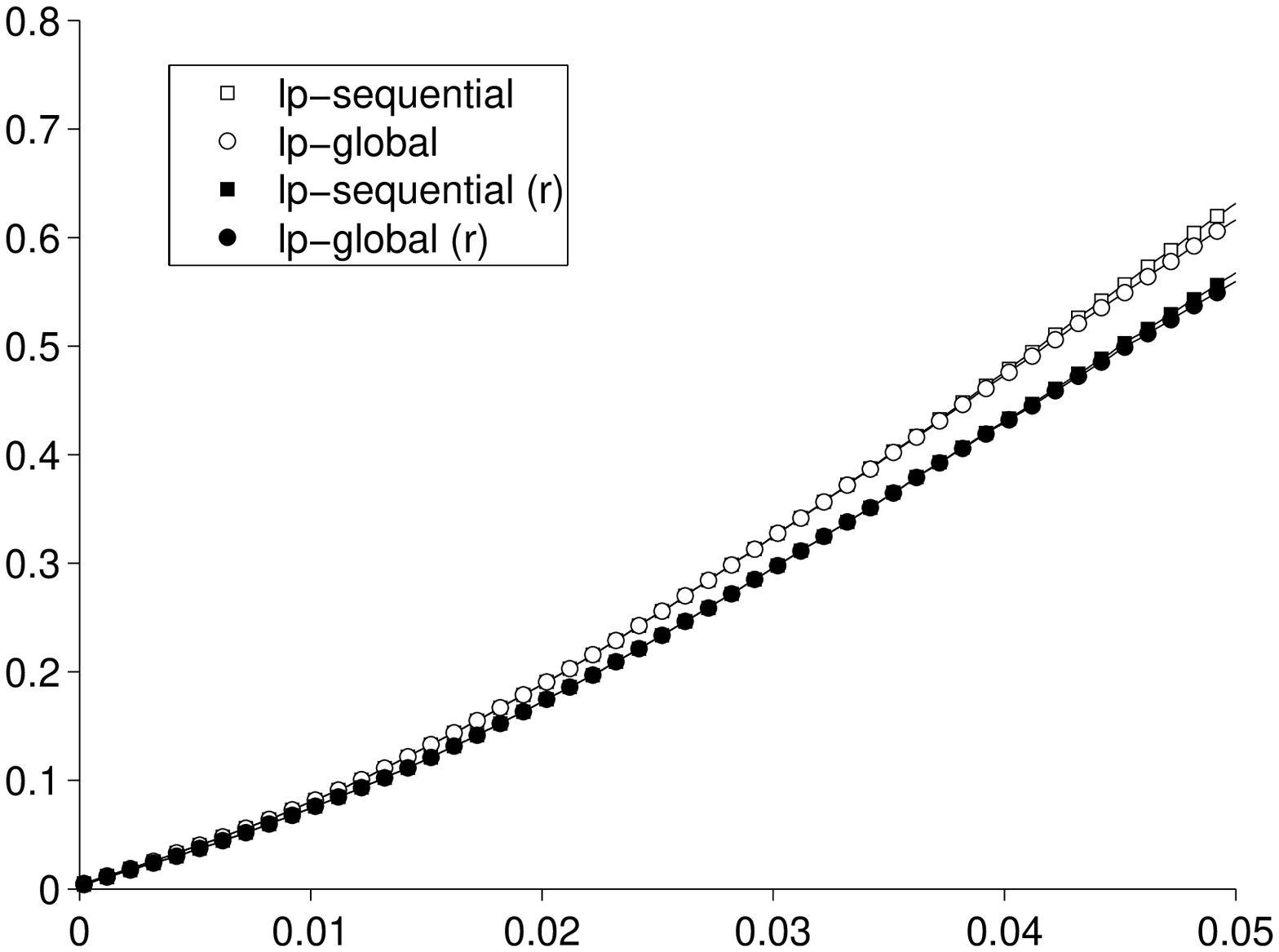, width=5.8cm, height=4.7cm} &
      \epsfig{file = 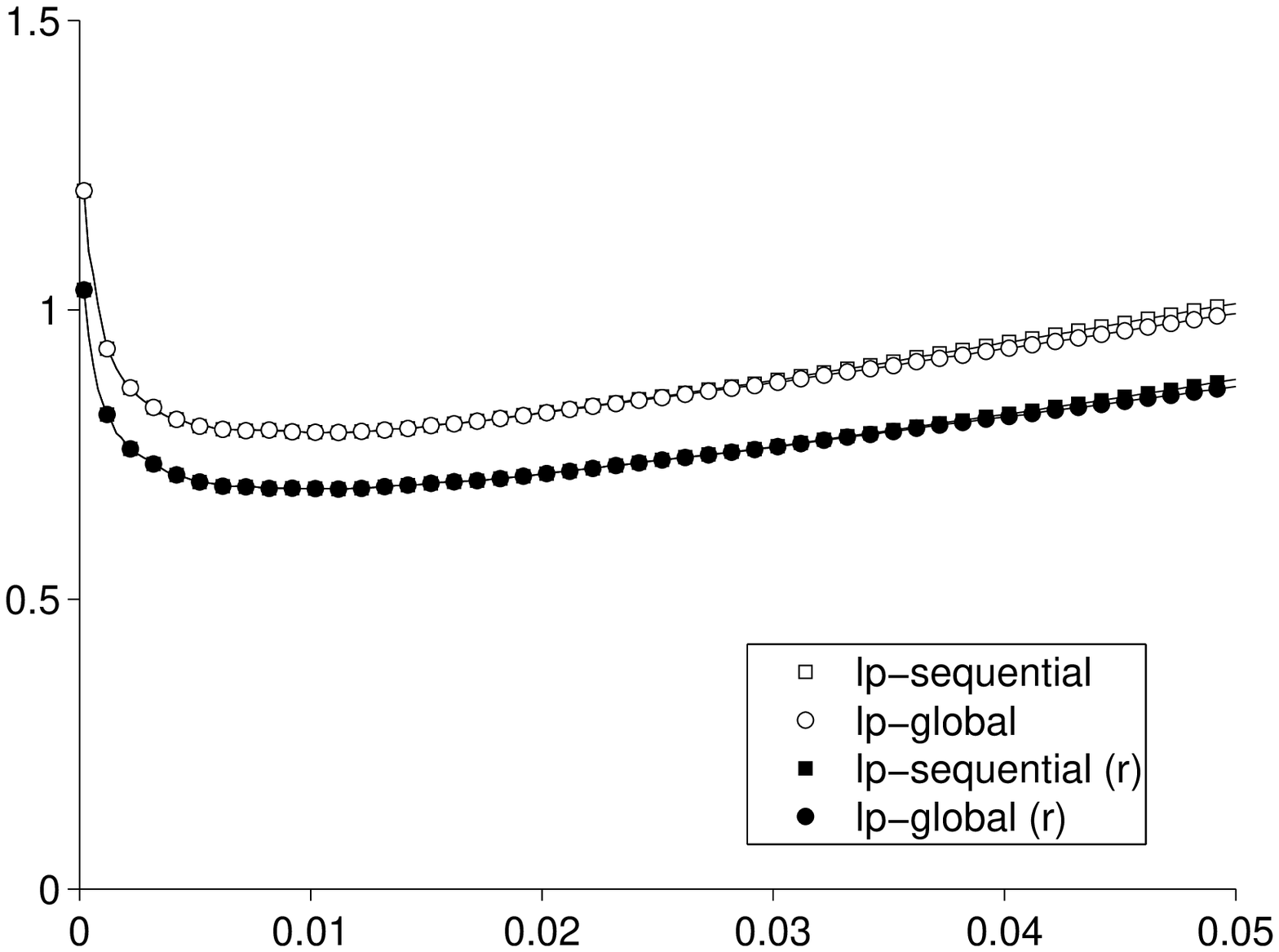, width=5.8cm, height=4.7cm} 
    \end{tabular*} \\[5ex]
    Simulation 5 \\[2ex]
    \epsfig{file = 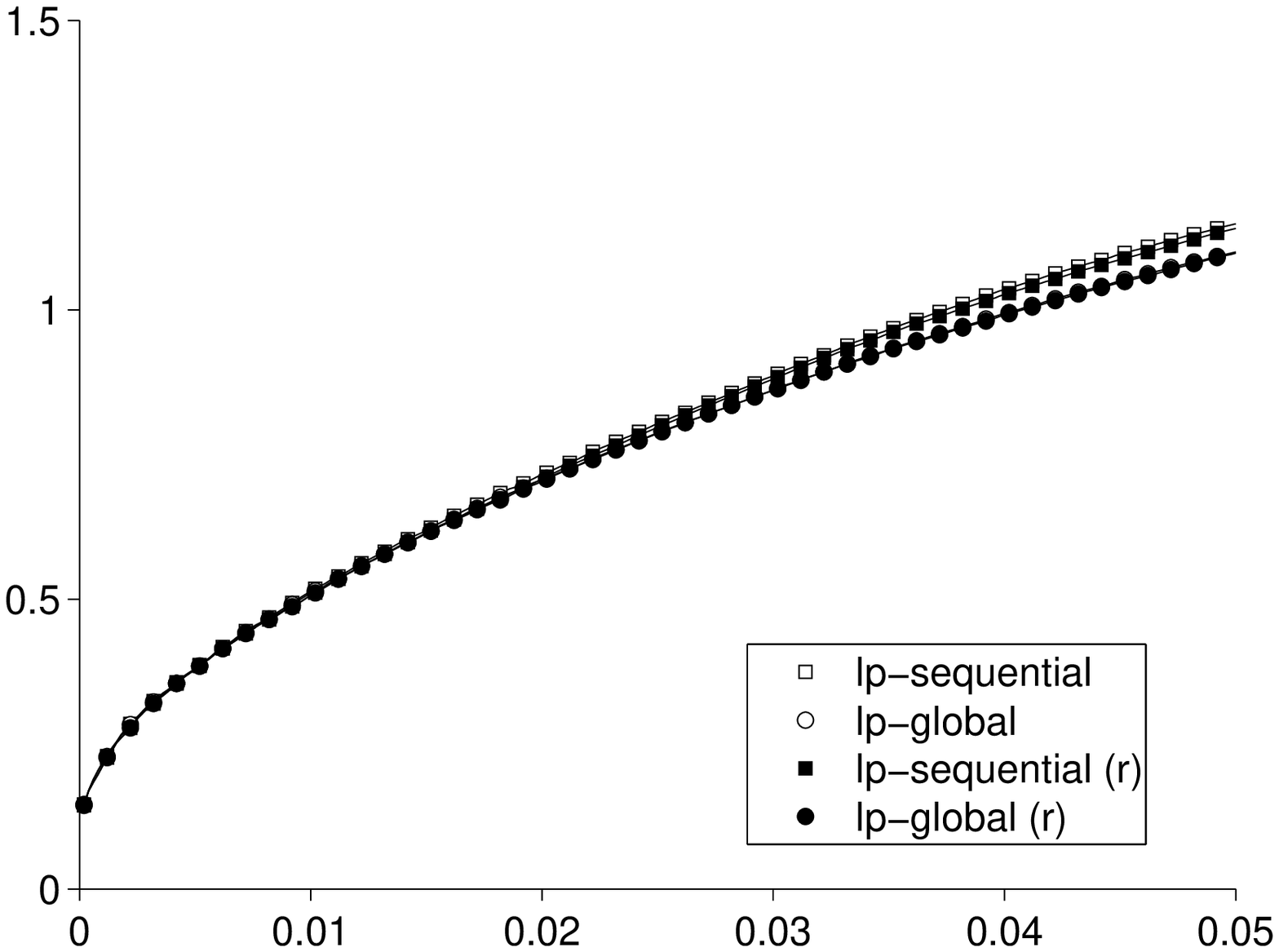, width=5.8cm, height=4.7cm} 
  \end{center}
  \caption{\rm\label{fig:pvals-2}
    Plots of $n\bar p\Sb i/i$ \emph{versus\/} $i/n$ in simulations 1--5,
    with $i/n\le 0.05$.  The plots with open symbols are those of
    $\pseq i$ and $\pglb i$ as in Figure \ref{fig:pvals}.  The plots
    with closed symbols are those of $\pseq i$ and $\pglb i$ computed
    with the extra constraint $c_1+\cdots+c_L\ge 0.9$.
  }
\end{figure}

\begin{figure}[ht]
  \renewcommand{\arraystretch}{0}
  \begin{center}
    \begin{tabular*}{.98\textwidth}{@{\extracolsep\fill} cc}
      \multicolumn{2}{c}{Simulation 1} \\[2ex]
      \epsfig{file = 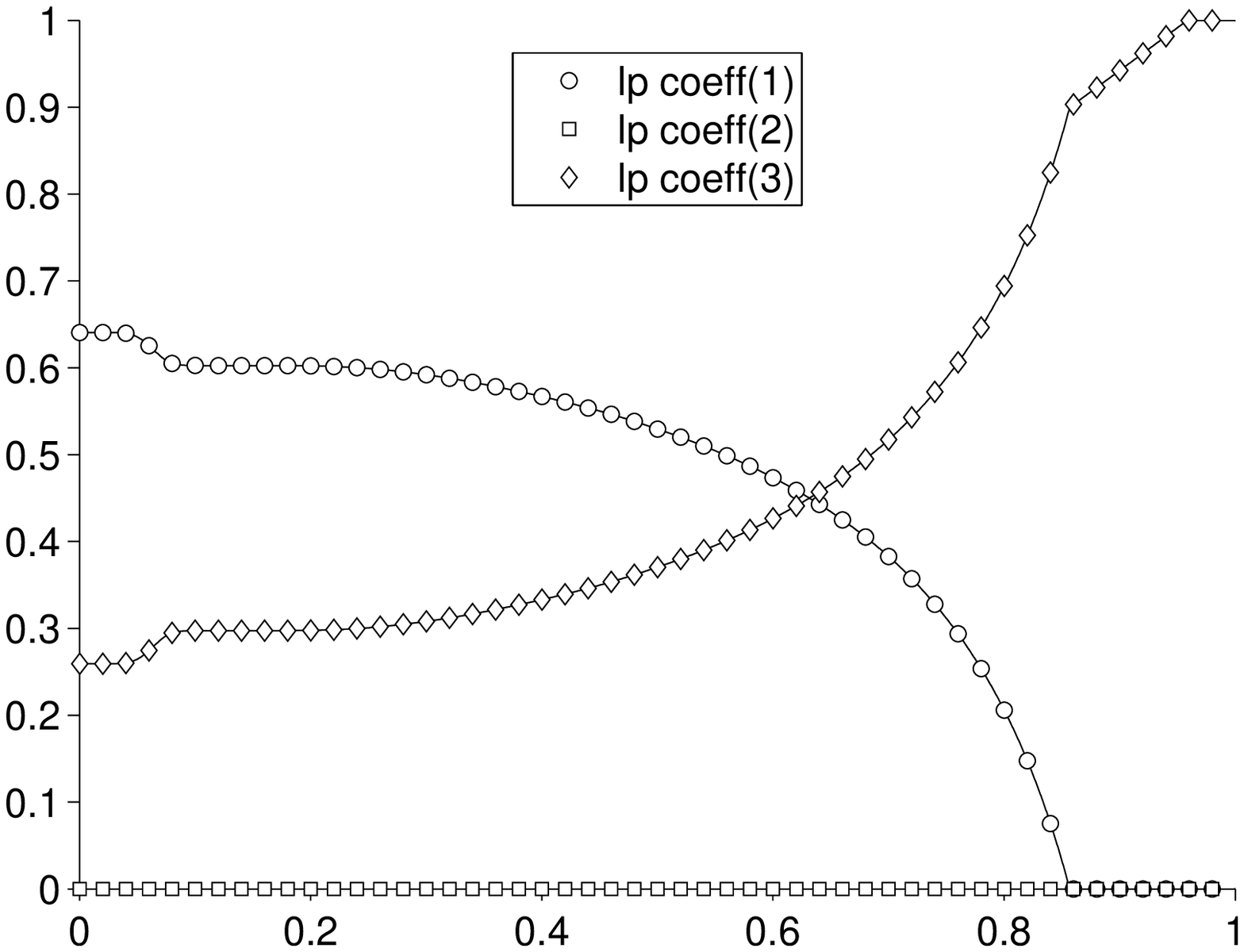,
        width=5.8cm, height=4.7cm} &
      \epsfig{file = 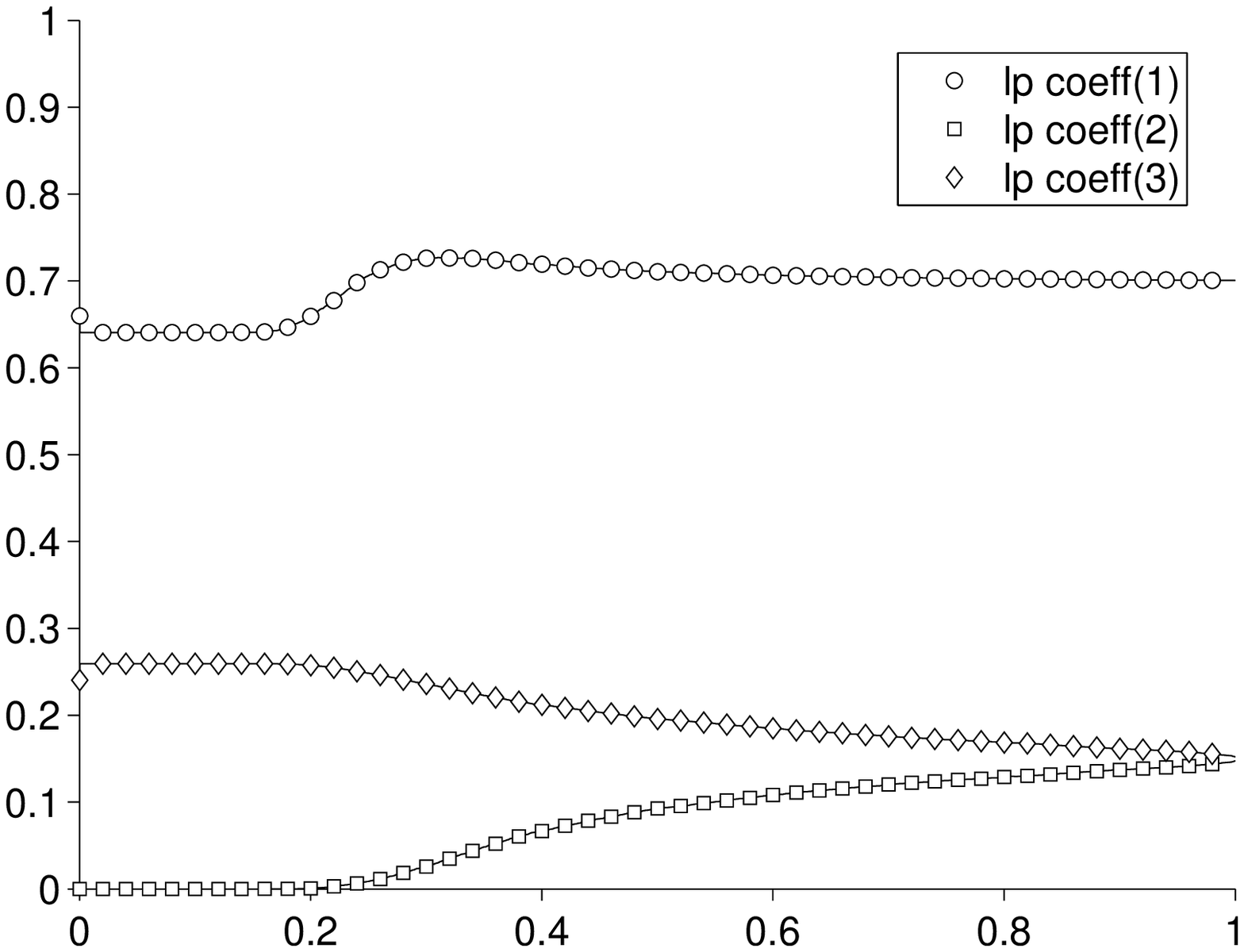,
        width=5.8cm, height=4.7cm} \\[5ex]
      \multicolumn{2}{c}{Simulation 5} \\[2ex]
      \epsfig{file = 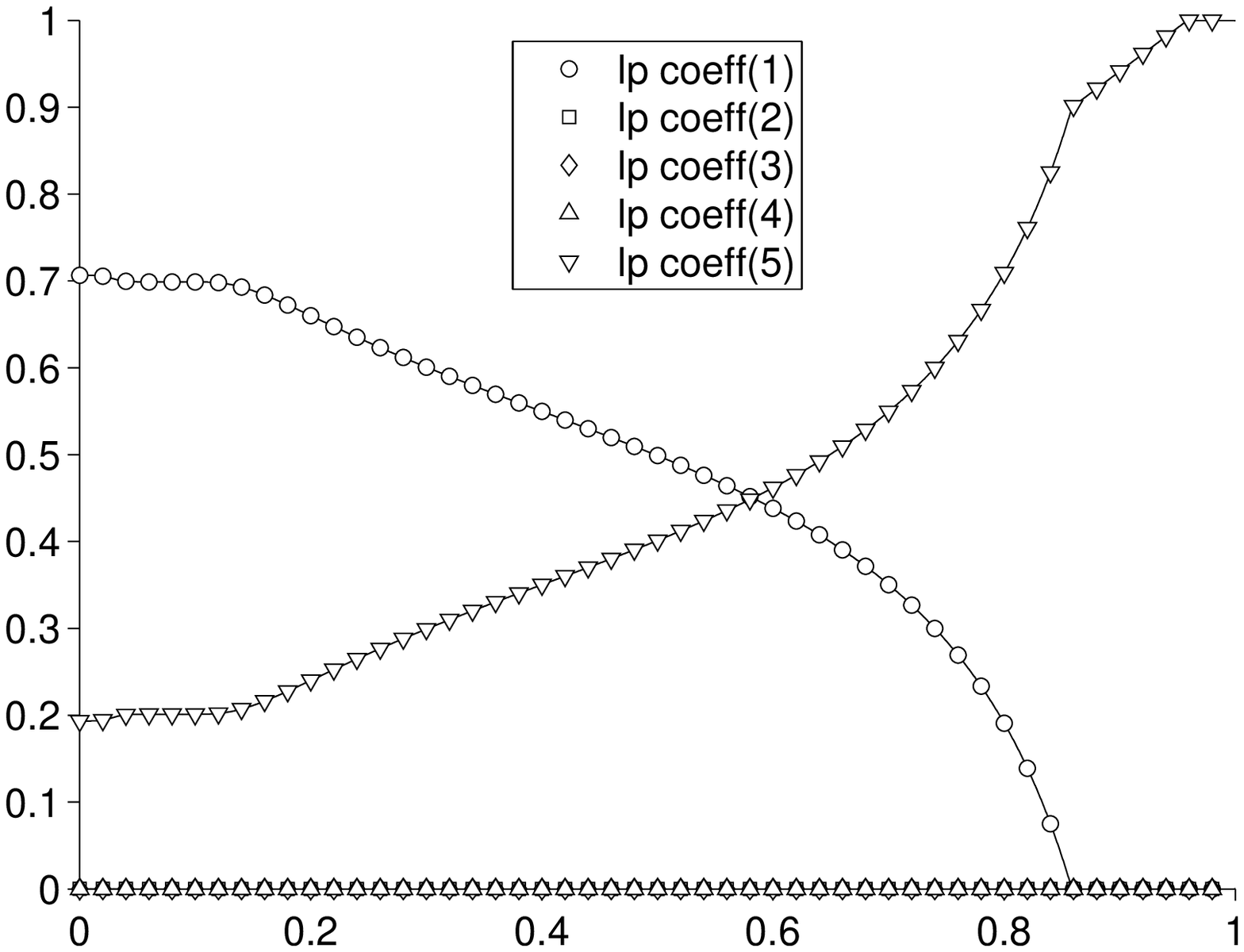,
        width=5.8cm, height=4.7cm} &
      \epsfig{file = 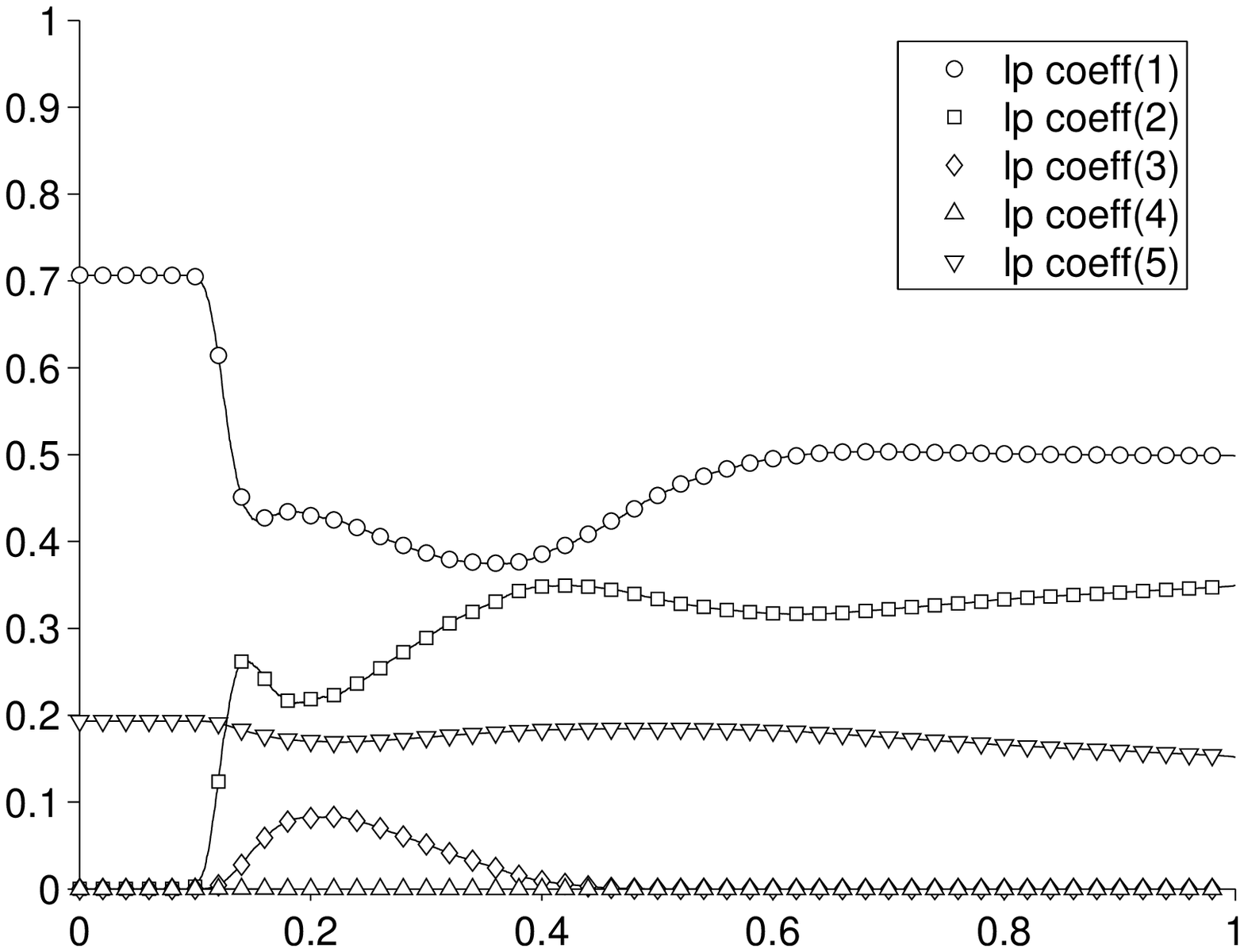,
        width=5.8cm, height=4.7cm} 
    \end{tabular*}
  \end{center}
  \caption{\rm\label{fig:coeff2}
    Plots of $c_{k,(i)}$ \emph{versus\/} $i/n$, $k=1,\ldots, n$ in
    simulations 1 and 5, where $c_{1,(i)}, \ldots, c_{L,(i)}$ are the
    coefficients to attain $\pseq{(i)}'$ (left) or $\pglb{(i)}'$
    (right), under the constraint $c_1+\cdots+c_L\ge 0.9$ in addition
    to those for $\pseq{(i)}$ and $\pglb{(i)}$ in
    Figure~\ref{fig:coeff}.
  }
\end{figure}

In Figure \ref{fig:pvals-2}, we compare the plots of $n p_{(i)}/i$ for
the $p$-values.  Since all rejections occur when $i\ll n$, we only
compare the plots with $i/n\le 0.05$.  It is seen that for small
$i/n$, the plots for $\pseq i$ and $\pglb i$ are very close to each
other, explaining why the performances of the BH procedure based on
the two types of $p$-value are similar.  Likewise, the plots for
$\pseq i'$ and $\pglb i'$ are very close to each other, and in
simulations 1--3, both are significantly lower than the plots of
$\pseq i$ and $\pglb i$, which explains the improved power yielded by
$\pseq i'$ and $\pglb i'$.  Finally, comparing Figures \ref{fig:coeff}
and \ref{fig:coeff2}, we can see that the extra constraint
$c_1+\cdots+c_L\ge 0.9$ substantially changes the plots of the
coefficients.  In particular, for $\pseq i'$ with $i\ll n$, the linear
programming sets two coefficients nonzero, as opposed to only one for
$\pseq i$.

From the above results, it is seen that the performances of the BH
procedure based on $\pseq i$ and $\pglb i$ are close to each other,
even though the latter one are subject to more constraints.  The
reason seems to lie in how $\pseq{(i)}$ are computed.  The evaluation
of $\pseq{(i)}$ incorporates the constraints imposed by $s\Sb j$ with
$j\ge i$.  For small $i$, the set of constraints is only different by
a small fraction from those that are imposed by the entire set of
$s\Sb j$.  Under regular conditions, constraints imposed by $s\Sb j$
with $j<i$ will not change the maximization substantially.  This
implies that for small $i$, $\pseq{(i)}$ and $\pglb{(i)}$ are close to
each other, as can be seen from Figure \ref{fig:pvals-2}.  Since the
BH procedure only reject nulls with small $p$-values, its performance
based on either type of $p$-values will be similar.

\section{MLE for prior probabilities of nulls}
\label{sec:mle}
Let $\Theta = \{\eno \theta L\}$.  As indicated in Section
\ref{sec:simul}, for composite nulls, in general the prior $\vf\pt$
may not be estimated consistently.  In this seciton, we consider under
what conditions $\vf\pt$ can be estimated consistently.  Under the
setup in Section~\ref{sec:intro}, suppose each $\PT_k$ has a density
$f_k$ and $\PA$ is absolutely continuous with respect to the
distribution under true nulls.  By the Radon-Nikodym theorem, $\PA$
has a density $\rho(x)\vf\pt\tp\vff(x)$ with $\rho(x)\ge 0$.  Then the
data $\eno X n$ are iid with density
\begin{align*}
  q(x) = [1-\pa + \pa \rho(x)]\vf\pt\tp\vff(x).
\end{align*}
Pretending all the nulls are true, the MLE for $\vf\pt$ is
\begin{align*}
  \hat{\vf\pt}_n = \mathop{\arg\sup}_{\vfc\in S}
  \sum_{i=1}^n \ln [\vfc\tp \vff(X_i)],
\end{align*}
where $S$ is a suitable set.  Usually, one would choose $S=\{\vfc \in
[0,1]^L: \sum c_k=1\}$ because by the definition of prior
probabilities, $\pt_k\ge 0$ and $\sum \pt_k =1$.  For the reason
described below, we shall make the setting a little more general.
Still suppose that the distribution under true nulls is a linear
combination of $\PT_k$.  However, now $\pt_k$ are allowed to be
negative.  In this setting, it had better merely regard $f_k$ as a
basis for a set of densities.  Then set
\begin{align} \label{eq:mle-S}
  S = \{\vfc: c_1+\cdots+c_L=1,\ \vfc\tp\vff\ge 0\}.
\end{align}

A reason for this choice of $S$ can be seen when density functions
under true nulls are linearly dependent.  In this case, it is
desirable to pick a basis from them, say $\eno f L$, and represent the
others as $g_j = \sum_k \lambda_{jk} f_k$.  By linear dependence,
$\lambda_{jk}$ can be negative.  Let the mixture density under true
nulls be $\vf a\tp\vff + \vf b\tp\vf g$, with $\sum a_k + \sum b_j=1$
and $a_k, b_j\ge 0$.  By representing it as $\vf\pt\tp \vff$, we get
$\pt_k = a_k + \sum_j b_j \lambda_{jk}$, which can be negative.  On
the other hand, $\sum\pt_k = 1$ and $\vf\pt\tp\vff\ge 0$.  Therefore,
$S$ in \eqref{eq:mle-S} contains $\vf\pt$.

Recall that if $A\subset\Reals^d$, its interior is $A^o = \{x:
B(\vfx,r)\subset A$ for some $r>0\}$, where $B(x,r) = \{z: |z_k -
x_k|< r, k=1,\ldots,d\}$.  By this definition, $S^o=\emptyset$.
However, regarding $S$ as a subset in $\{\vfc: \sum c_k=1\}$, we have
$S^o=\{\vfc:$ for some $r>0$, $\vfc+\vfv\in S$ 
$\forall \vfv\in B(\vf0,r)$ with $\sum v_k=0\}$.  Both $S$ and $S^o$ are
convex. Since $S$ contains all $\vfc$ with $c_k\ge 0$ and $\sum
c_k=1$, $S^o\not=\emptyset$.
\begin{prop} \label{prop:mle}
  Suppose $\int q|\ln f_k|<\infty$ and $\eno f L$ are linearly
  independent.  Let $\pa\in (0,1)$.  If $\vf\pt\in S^o$, then
  \begin{align*}
    \hat{\vf\pt}_n\convP \vf\pt \iff
    \int \rho f_k =1 \text{ for all $k$}.
  \end{align*}
\end{prop}

Apparently, if $\rho=1$, then $\int \rho f_k = 1$.  A question is
whether nontrivial $\rho\ge 0$ satisfying the condition exists.  Since
$\int \rho(f_k - f_1)=0$, 
provided $f_k\in L^2$, one might search for $\rho$ among functions in
$L^2$ that are orthogonal to $f_k-f_1$.  However, such functions are
not always nonnegative.  Moreover, oftentimes $f_k\not\in L^2$.
The construction below avoids these potential problems and seems to be
general.
\begin{example}\rm
  We only consider how to construct $\rho\ge 0$ that are unbounded on
  $E = \{x: \vf\pt\tp\vff(x)>0\}$.  The general case follows the same
  idea.  The main step is to find bounded $\eno \psi L\in
  C(\Reals^d)$, such that the $L\times L$ matrix $M = (M_{ik})$ is
  nonsingular, where $M_{ik} = \int \psi_i f_k$.  Once this is done,
  to construct $\rho$, fix $\phi\ge 0$ continuous such that $\int\phi
  f_k<\infty$ and $\sup_{x\in E}\phi(x)=\infty$.  Such $\phi$ always
  exist.  By $\det M\not=0$, there are unique $\eno a L\in\Reals$,
  such that $\sum a_i M_{ik} = 1-\int \phi f_k$ for each $k$.  Then
  $\int h f_k=1$, where $h = \phi + \sum a_i \psi_i$.  It is easy to
  see $h\in C(\Reals^d)$ is lower bounded and $\sup_{x\in E} h(x) =
  \infty$.  Then for $c>0$ small enough, $\rho=1-c+ch\in C(\Reals^d)$
  is nonnegative with $\sup_{x\in E} \rho(x) = \infty$ and $\int \rho
  f_k = 1-c + c\int h f_k = 1$.

  To see that $\eno\psi L$ as above exist, recall
  \begin{align*}
    \det M
    &= \sum_\sigma \sgn(\sigma) \prod_k \int f_{\sigma(k)}
    \psi_k \\
    &=
    \int \sum_\sigma \sgn(\sigma) \prod_k f_{\sigma(k)}(x_k)
    \psi_k(x_k)\, d\vfx \\
    &=
    \int \prod_k \psi_k(x_k) \det[f_i(x_k)]\,d\vfx.
  \end{align*}
  where the sum is over all permutations $\sigma$ of $1,\ldots, L$ and
  $\sgn(\sigma)$ is the sign of $\sigma$.  Denote $D(\vfx) =
  \det[f_i(x_k)]$.  Since $|D(\vfx)| \le \sum_\sigma \prod_k
  f_{\sigma(k)}(x_k)$, $D\in L^1$.  Because $\eno f L$ are linearly
  independent, we claim
  \begin{align}\label{eq:linind}
    \ell(\vfx: D(\vfx)=0)\not=0,
  \end{align}
  where $\ell$ is the Lebesgue measure.  If \eqref{eq:linind} holds,
  then the characteristic function of $D$ is nonzero.  Therefore,
  there are $\eno t L\not=0$, such that $\int e^{i(t_1x_1+\cdots+t_L
    x_L)} D(\vfx)\,d\vfx \not=0$.  It follows that there are
  $\psi_k(x)$ of the form  $\sin(t_k x)$ or $\cos(t_k x)$,
  such that $\det M\not=0$.

  We use induction to prove \eqref{eq:linind}.  For $L=2$, if
  $D(\vfx)=0$ a.e., then $f_1(x_1)f_2(x_2) = f_1(x_2)f_2(x_1)$, a.e.
  Integrating over $x_2$ yields $f_1(x_1)=f_2(x_1)$ a.e.,
  contradicting the assumption that $f_1$ and $f_2$ are linearly 
  independent. 

  For $L>2$, suppose \eqref{eq:linind} holds for $L-1$ linearly
  independent $f_i$.  Now
  \begin{align*}
    D(\vfx) = \sum_{i=1}^L (-1)^{L+i} f_i(x_L) M_i(\eno x{L-1}),
  \end{align*}
  where $M_i(\eno x {L-1})$ is the determinant of the $(L-1)\times
  (L-1)$ matrix consisting of $f_l(x_k)$, $l\not=i$, $k=1,\ldots,
  L-1$.  Given $\eno x{L-1}$, $D(\vfx)$ is a linear
  combination of $f_i(x_L)$.  Therefore, if $D(\vfx)=0$ a.e., then, by
  the linear independence of $f_i(x)$, $M_i(\eno x{L-1})=0$ a.e.\ for
  each $i=1,\ldots, L$.  However, this contradicts the induction
  hypothesis.  \qed
\end{example}

\section{Discussion}\label{sec:discussion}
In the article, we have focused on the case of finitely composite
nulls, where true nulls are only associated with a finite number of
distributions.  Formally, it is straightforward to generalize the
constrained maximization to the case of infinitely composite nulls.
However, usually the maximization will involve infinitely many
degrees of freedom and it becomes unclear how to accommodate this with
a finite number of observations.  A more direct approach might be to
partition the set of null distributions into a finite number of
subsets and use the envelopes of the subsets to compute $p$-values.
To be more specific, given a partition $\eno\Theta L$ of $\Theta$, let
$u_k(t) = \sup_{\theta\in\Theta_k} \phi_\theta(t)$ and $l_k(t) =
\inf_{\theta\in\Theta_k} \phi_\theta(t)$.  Then define, for example,
$M_n(t) = \sup\{\vfc\tp\vf u(t): \vfc \in \Delta,\ \vfc\tp\vf l(t)$ is
dominated by $\EP_n(t)$ up to a small margin$\}$.  Unfortunately, some
of the constraints available to the finitely composite case can no
longer be used.  Another issue is how to select the partition.  Too
coarse partition will only yield loose constraints on $c_k$ and too
fine partition will result in many degrees of freedom.  Either way,
the obtained $M_n(t)$ may not be much different from the unconstrained
maximum probability.

As is known, FDR control can be realized by the local FDRs
\cite{efron:etal:01}.  For the simple case, the local FDR at $x$ is
$(1-\pa) f_0(x)/h(x)$, where $\pa$ may be replaced with 0, $f_0$ is
the density under true nulls, and $h$ is the overall density of the
data $\eno X n$ or an estimate of the density.  For the finitely
composite case where the null distributions have densities $\eno
f L$, we may derive a conservative estimate of the local FDR by
$\rho(x)/h(x)$, where
\begin{align*}
  \rho(x) 
  = \sup\{
  \vfc\tp \vff(x): \ \vfc\in \Delta \text{ and } 
  \vfc\tp \vff \le h\}.
\end{align*}
Alternatively, if the dimension of $X_i$ is high, then we may work on
$s_i = s(X_i)$, with the local FDR defined as $\rho(s_i)/f(s_i)$,
where $h$ is now the overall density of $\eno s n$ or an estimate,
while
\begin{align*}
  \rho(t) = \sup\{
  \vfc\tp \vf\phi(t): \ \vfc\in \Delta \text{ and }
  \vfc\tp \vf\phi \le h
  \}.
\end{align*}
It is worth pointing out that, unlike the simple case, the BH
procedure based on $M_n(t)$ and the FDR control based on
$\rho(x)/h(x)$ are no longer equivalent.  The reason is that $M_n(t)$
is of the form $\max_\vfc \int \vfc\tp\vf\phi$.  The density of
$M_n(t)$, if existent, in general is different from $\max_\vfc
\vfc\tp\vf\phi$ that is associated with the local FDR.  It remains to
be seen how much difference the two approaches may have.

\setcounter{section}{0}
\renewcommand{\theequation}{A.\arabic{equation}}
\renewcommand{\thesection}{A.\arabic{section}}
\renewcommand{\thesubsection}{A.\arabic{subsection}}

\def\IInt{\int}

\newcounter{XX}
\section*{Appendix}
In this section, we give proofs of the theoretical statements of the
article.  The Lebesgue measure on $\Reals^d$ will be denoted by
$\ell$.  For any nondecreasing function $f$ defined on $\Reals$ and
$x\in\Reals$, if $A:=\sup\{t: f(t) \le x\}\not=\emptyset$, define
$f^*(x)=\sup A$, otherwise, define $f^*(x)=-\infty$.  By this
definition, if $f$ is left-continuous and $x\in f(I)$, then
$f(f^*(x))=x$.

\section{Proofs for Section \ref{sec:max}}
\begin{proof}[Proof of Proposition \ref{prop:equivalence}]
  Since $s_i=-\infty$ $\iff$ $X_i\in D_t$ for all $t$, by D3 and the
  random mixture model, the probability of the event is 0, hence
  proving 1).  By the right-continuity of $D_t$,
  \begin{align*}
    s_i \le t \iff X_i\in D_s \text{ for all $s>t$} \iff X_i \in D_t,
  \end{align*}
  yielding 2).  By $P(s_i\le t) = P(X_i\in  D_t) = \phi_\theta(t)$, 3)
  holds and 4) follows from 3) and the random mixture model.  To get
  5), given $t$, for any $\epsilon>0$, there is $\theta\in\Theta$ such
  that $M(t)\le \phi_\theta(t)+\epsilon$.  By D3, $M(s)\ge
  \phi_\theta(s)\to \phi_\theta(t)$ as $s\uto t$, giving $M(t-) +
  \epsilon \ge M(t)$.  Since $M$ is nondecreasing and $\epsilon$ is
  arbitrary, this implies $M(t-) = M(t)$.
\end{proof}

\begin{proof}[Proof of Proposition \ref{prop:FDR}]
  To see that Procedures \ref{proc:max} and \ref{proc:equivalent} are
  the same, by Proposition \ref{prop:equivalence},
  \begin{gather*}
    \text{Procedure \ref{proc:max} \textit{accepts\/} } H_i
    \!\iff\!
    s_i > \tau
    \!\iff\!
    \frac{M(t)}{\alpha} > \frac{R_n(t)}{n}\
    \,\forall t\ge s_i.
  \end{gather*}
  Because $M(t)$ is nondecreasing and $R_n(t)$ is an nondecreasing
  step function that has jumps only at $s_i$,
  \begin{gather*}
    \text{Procedure \ref{proc:max} \textit{accepts\/} } H_i
    \!\iff\!
    \frac{M(s_j)}{\alpha} > \frac{R_n(s_j)}{n}\
    \forall\, s_j\ge s_i.
  \end{gather*}
  Taking into account the possibility of ties, it is not hard to see
  that the condition on the right hand side is equivalent to $s_i>
  s\Sb R$, which implies Procedures \ref{proc:max} and
  \ref{proc:equivalent} always reject the same set of nulls.
  
  By the random  mixture model, for $X_i$ under true nulls, the
  distribution of $\PT_{\eta_i}(D_{s_i})$ is a mixture of those of
  $\phi_\theta(s(X))$ under $F_\theta$, $\theta\in\Theta$.  By
  Proposition \ref{prop:equivalence}, under $F_\theta$,
  $\phi_\theta(s(X)) \sim \unif(0,1)$.  Therefore, for
  $X_i$ under true nulls, $s_i$ are iid $\sim \unif(0,1)$.

  Procedure \ref{proc:equivalent} is the BH procedure applied to
  $M(s_i)$.  Since $M(s_i)\ge F_{\eta_i}(D_{s_i})$, under true nulls,
  $P(M(s_i)\le x) \le P(F_{\eta_i}(D_{s_i})\le x) =  x$.  The proof
  then follows from Theorem~5.1 and the comment that follows in
  \cite{benjamini:yek:01}.
\end{proof}

\section{Proofs for Section \ref{sec:lp-bh}}
\label{sec:append-lp-bh}
First, note that for Procedures \ref{proc:finite} and
\ref{proc:finite2}, the number of rejections and that of false
rejections are $R = R_n(\tau)$ and $V = V_n(\tau)$, respectively.

\begin{proof}[Proof of Lemma \ref{lemma:Mt-property}]
  Let $s<t$.  Then $\CA_{n,s}\subset\CA_{n,t}$ and $\vfc\tp \vf\phi(s)
  \le \vfc\tp\vf\phi(t)$ for any $\vfc\in\Delta$, giving $M_n(s)\le
  M_n(t)$.  Thus $M_n$ is nondecreasing.  Next suppose
  $\phi_i\in C(\Reals)$ for all $i$.

  1) Given $t$, as $0<t-u\ll 1$, $[u,t)$ has no point in
  $\CT_n$ and, almost surely, no $s_i$.  Thus $\CA_{n,u} = \CA_{n,t}$.
  Let $K=\{\vfc\in\Delta:\  \vfc\tp \vf\phi\in \CA_{n,t}\}$.  It is
  seen that $K$ is compact and $\vfc\tp\vf\phi(s)$ is a uniformly
  continuous function in $(\vfc, s)\in K\times \CI$.  Then
  $\sup_{\vfc\in K} \vfc\tp \vf\phi(s)$ is continuous in $s$, yielding
  $M_n(u)\to M_n(t)$ as $u\uparrow t$.  Thus $M_n$ is left-continuous.

  2) Since $M_n$ is nondecreasing, $M_n$ has a right-hand limit at
  every $t$.  It only remains to be shown that at every
  $t\not\in\{s_1,\ldots, s_n\}$, $M_n$ is right-continuous.  Now, as
  $0<u-t\ll 1$, $[t,u)$ contains no point in $\CT_n$ and no $s_i$,
  yielding $\CA_{n,u}=\CA_{n,t}$.  Then the right-continuity follows
  from the same argument for the left-continuity.
\end{proof}

In addition to Lemma \ref{lemma:Mt-property}, we need a few lemmas to
prove Theorem \ref{thm:lp-sequential}.  For $t\in \CI$, define
$\sigma$-field
\begin{align*}
  \CF_t = \CF(R_n(t-),\, V_n(t-),\, R_n(s),\, V_n(s):\, s\ge t).
\end{align*}
Then $\{\CF_t, t\in\CI\}$ is a backward filtration, i.e.,
$\CF_t\subset\CF_s$ for $t>s$.

\begin{lemma}\label{lemma:At-measurable}
  Suppose $\phi_i\in C(\Reals)$ for all $i$.  Then for $t\in\Reals$,
  $M_n(t)$ is $\CF_t$-measurable.
\end{lemma}
\subsection*{\sc Proof}
It suffices to show that given $a\ge 0$,
$\{M_n(t)\le a\}\in \CF_t$ for $t\in\CI$.  For $\vfc\in\Delta$, 
$\vfc\tp\vf\phi\in C(\Reals)$ and $\{\vfc\tp\vf\phi\in \CA_{n,t}\} =
E_1\cap E_2$, where
\begin{align*}
  E_1
  &= \Cbr{\vfc\tp\vf\phi(s_i)\le \EP_n(s_i) + \epsilon_n
    \text{ for } s_i\ge t
  },\\
  E_2
  &=\Cbr{\!\!
    \begin{array}{c}
      \EP_n(t_2)-\EP_n(t_1) \ge  
      \vfc\tp[\vf\phi(t_2)-\vf\phi(t_1)] - \epsilon_n \text{ for}
      \\[.5ex]
      t_i \in \CT_n \cap [t,\tau_2] \text{ with } t_1<t_2
    \end{array}
    \!\!
  }.
\end{align*}

Note $E_1=\{\vfc\tp\vf\phi(s)\le R_n(s)/n +  \epsilon_n$ $\forall
s\ge t$ with $R_n(s)>R_n(s-)\}$.  Since $R_n(s-)\in \CF_t$ for $s\ge
t$, it can be seen that $E_1\in \CF_t$.  On the other hand, $E_2\in
\CF_t$.  Therefore, $\{\vfc\tp\vf\phi\in \CA_{n,t}\}\in \CF_t$.

Since $\vfc\tp\vf\phi\in\CA_{n,t}$ implies  $\vfr\tp\vf\phi\in
\CA_{n,t}$ for any $\vfr\in \Rats^L\cap\Delta$ with $r_i\le c_i$,
where $\Rats$ is the set of rational numbers, $M_n(t) =
\sup\{\vfr\tp\vf\phi(t): \vfr\in \Rats^L\cap\Delta,\
\vfr\tp\vf\phi\in\CA_{n,t}\}$.  Notice that $\vfr\tp\vf\phi(t)$ is
nonrandom.  Then
\begin{align*}
  \Cbr{M_n(t)\le a} 
  =
  \bigcap_{
    \substack{\vfr\in \Rats^L\cap \Delta \ \rm s.t.\\
      \vfr\tp\vf\phi(t)>a}
  }\{\vfr\tp\vf\phi\not\in \CA_{n,t}\}\in\CF_t.  \eqed
\end{align*}

The next goal is to show $\tau$ is a stopping time of the backward
filtration $\CF_t$.  If $\sup\CI=\infty$, then $\tau$ has to start at
$\infty$.  To get around this problem, we use truncations.  Let
$\CI_R$ be as in Procedure \ref{proc:finite}.  Given $c<\sup\CI$, define
\begin{align*}
  \CI_c = \CI_R\cap (-\infty, c],
  \quad
  \tau_c =
  \begin{cases}
    \sup\CI_c & \text{if}\ \CI_c\not=\emptyset\\
    \inf\CI & \text{otherwise}
  \end{cases}
\end{align*}
\begin{lemma}\label{lemma:tau0}
  As $c\uto \sup\CI$, $\tau_c\uto \tau$ a.s.
\end{lemma}
\begin{proof}
  It suffices to show $\tau<\sup\CI$ a.s.  By definition, $\tau\le
  \sup\CI$.  The event $\{\tau=\sup\CI\}$ implies there are
  $t_k\uto \sup\CI$, such that $M_n(t_k)\le \alpha[R_n(t_k)\vee 1]/n$.
  By Lemma \ref{lemma:Mt-property}, $M_n(t_k)\to M_n(\sup\CI)=1$ a.s.
  On the other hand, $[R_n(t_k)\vee 1]/n\le 1$.  Therefore,
  $P(\tau=\sup\CI)=0$.
\end{proof}

\begin{lemma}\label{lemma:tau}
  Suppose $\phi_k\in C(\Reals)$.  Then 1) there is
  $t_0>\inf\CI$ such that for any $c\in \CI$, $\tau_c \ge t_0$, 2)
  for $c\in\CI$, $\tau_c$ is a stopping time of the backward
  filtration $\{\CF_t,\, t\in (\inf\CI, c]\}$.
\end{lemma}
\begin{proof}
  1) Let $u(t):=\phi_1(t) + \cdots + \phi_L(t)$.  Then $M_n(t)\le
  u(t)$ and
  \begin{align*}
    \tau_c
    &
    \ge \sup\Cbr{t\in (\inf\CI, c]:\ \frac{u(t)}{\alpha} \le
      \frac{R_n(t)\vee 1}{n}
    } \\
    &
    \ge t_0:=\sup\Cbr{t\in(\inf\CI, c]:\ \frac{u(t)}{\alpha} \le 
      \frac{1}{n}}.
  \end{align*}
  Since $\phi_k\in C(\Reals)$ and $\phi_k(t)\to 0$ as $t\to\inf\CI$,
  the set on the right hand side is nonempty, yielding $t_0>\inf\CI$.

  2) By definition, $\tau_c$ is a stopping time of the backward
  filtration $\CF_t$ if $\{\tau_c\ge t\}\in \CF_t$ for every $t\in
  (\inf\CI, c]$.  Denote $E=\{\tau_c\ge t\}$.  We first show
  \begin{align} \label{eq:stop-E}
    E = \Cbr{\exists s\in [t,c] \text{ such that } 
    \frac{M_n(s)}{\alpha} \le \frac{R_n(s)\vee 1}{n}
    }.
  \end{align}

  The right hand side of \eqref{eq:stop-E} equals $\{\CI_c\cap[t,c]
  \not= \emptyset\}$, which is a subset of $E$. 
  On the other hand, the difference between the two events is
  \begin{align*}
    &
    \{\tau_c\ge t,\ \CI_c\cap [t,c]=\emptyset\}\\
    =\
    &
    \{\CI_c\not=\emptyset,\ \CI_c\cap [t,c]=\emptyset, \ \sup\CI_c\ge
    t\} \\
    \subset\
    &
    \Cbr{
      \frac{M_n(t)}{\alpha}> \frac{R_n(t)\vee 1}{n},\
      \exists
      t_k\uparrow t \text{ with }
      \frac{M_n(t_k)}{\alpha} \le \frac{R_n(t_k)\vee 1}{n}
    }.
  \end{align*}
  Since by Lemma \ref{lemma:Mt-property} $M_n$ is left-continuous,
  $M_n(t_k)\to M_n(t)$.  On the other hand, $R_n(t_k) \to R_n(t-)\le
  R_n(t)$.  Thus, the last event is empty and \eqref{eq:stop-E} holds.
  Note that by similar argument,
  \begin{align} \label{eq:M-tau}
    M_n(\tau_c)/\alpha \le [R_n(\tau_c)\vee 1]/n.
  \end{align}

  Let $A=\{M_n(t)/\alpha \le [R_n(t)\vee 1]/n\}$.  Then $A\subset E$
  and $A\in \CF_t$.  We next show $E = A\cup\Gamma$, where
  $\Gamma=\bigcap_{k=1}^\infty\bigcup_{r\in \Rats\cap (t,c]}
  \Gamma_{r,k}$, with
  \begin{align*}
    \Gamma_{r,k} = \Cbr{
      \frac{M_n(r)}{\alpha} \le \frac{R_n(r+1/k)\vee 1}{n}
    }.
  \end{align*}
  Once this is done, by $M_n(r)\in \CF_r$ (cf.\ Lemma
  \ref{lemma:At-measurable}) and $R_n(r+1/k)\in \CF_r$,
  $\Gamma_{r,k}\in \CF_r \subset \CF_t$ for any $r>t$.  Then $E\in \CF_t$.

  Note $E-A$ implies $\tau_c>t$, which in turn implies there are
  $r_k\in \Rats$ with $t<r_k<\tau_c<r_k+1/k$.  By
  \begin{align*}
    \frac{M_n(r_k)}{\alpha}
    \le 
    \frac{M_n(\tau_c)}{\alpha}
    \le
    \frac{R_n(\tau_c)\vee 1}{n}
    \le
    \frac{R_n(r_k+1/k)\vee 1}{n},
  \end{align*}
  $\Gamma_{r_k,k}$ holds for all $k$.   Thus $E-A\subset\Gamma$.

  It remains to show that $\Gamma\subset E$.  Suppose there are
  $r_k\in \Rats\cap (t,c]$ with $M_n(r_k)/\alpha \le [R_n(r_k+1/k)\vee
  1]/n$.  Then $r_k$ has a subsequence, say, itself, that
  converges to some $s\in [t,c]$.  Since $M_n$ is nondecreasing and
  left-continuous, while $R_n$ is nondecreasing and right-continuous,
  \begin{align*}
    \frac{M_n(s)}{\alpha} \le \Linf_k \frac{M_n(r_k)}{\alpha}
    \le \Lsup_k \frac{R_n(r_k+1/k)\vee 1}{n} \le
    \frac{R_n(s)\vee 1}{n}.
  \end{align*}
  Therefore $s\in \CI_c$ and $\CI_c\cap [t,c]\not=\emptyset$.  Thus
  $\Gamma\subset E$.
\end{proof}

\begin{lemma}\label{lemma:empirical}
  For $n\ge 1$, denote
  \begin{align} \label{eq:Gamma}
    \Gamma_n = \Cbr{(1-\pa)\vf\pt\tp\vf\phi \in \CA_{n,t},
      \ \forall t\in\CI}.
  \end{align}
  Suppose $\PD$ is continuous.  Then, provided
  $\exp(-2n\epsilon_n^2)\le 1/2$,
  \begin{align*}
    P(\Gamma_n) \ge 1- (1+|\CT_n|)\exp\Cbr{-2 n\epsilon_n^2}.
  \end{align*}
\end{lemma}
\begin{proof}
  Since $\PD$ is continuous, by the DKW inequality \cite{massart:90},
  for $\lambda>0$ and $n\ge 1$, as long as $\exp(-2n\lambda^2)\le
  1/2$,
  \begin{align*}
    P\Cbr{\sup (\PD-\EP_n) \ge  \lambda
    } 
    \le
    \exp(-2n\lambda^2).
  \end{align*}
  By $\PD(t) = (1-\pa)\vf\pt\tp \vf\phi(t) + \pa
  \PA(D_t)$, 
  \begin{align*}
    P\Cbr{(1-\pa)\vf\pt\tp\vf\phi(t) \ge \EP_n(t)+\lambda \text{ for
        some } t}
    \le \exp(-2n\lambda^2).
  \end{align*}

  DKW inequality also implies that, given $x\in\Reals$,
  \begin{align} \label{eq:DKW-modified}
    P\Cbr{\sup_{t\ge x} \Cbr{[\PD(t)-\PD(x)] - [\EP_n(t)-\EP_n(x)]}
      \ge \lambda} 
    \le \exp(-2n\lambda^2).
  \end{align}
  Assuming \eqref{eq:DKW-modified} is true for now, it follows that
  \begin{align*}
    P\Cbr{\!\!
      \begin{array}{c}
        \displaystyle
        \PD(t) - \PD(t_i) \ge \EP_n(t) - \EP_n(t_i)+\lambda
        \\[.5ex]
        \text{for some } t_i\in\CT_n \text{ and } t>t_i
      \end{array}
      \!\!} \le
    |\CT_n|\exp(-2n\lambda^2).
  \end{align*}
  Since $\PD(t) - \PD(t_i) \ge (1-\pa)\vf\pt\tp[\vf\phi(t) -
  \vf\phi(t_i)]$ for $t>t_i$, by letting $\lambda=\epsilon_n$, the
  Lemma then follows.

  Finally, to get \eqref{eq:DKW-modified}, let $y=\PD(x)$.  By
  quantile transformation,
  \begin{align*}
    &
    \sup_{t\ge x} \{[\PD(t)-\PD(x)] - [\EP_n(t)-\EP_n(x)]\} \\
    \sim\
    &
    \xi=\sup_{s\ge y} \{s-y - [\bbG_n(s)-\bbG_n(y)]\},
  \end{align*}
  where $\bbG_n$ is the empirical distribution of $U_i=Q(X_i)$.
  Since $U_i$ are iid $\sim\unif(0,1)$, $V_i = U_i-y + \cf{U_i\le y}$
  are iid $\sim\unif(0,1)$ as well and $\xi = \sup_{0\le s\le 1-y}
  [s-\bbG_n'(s)]$, where $\bbG_n'$ is the empirical distribution of
  $V_i$.  Applying DKW inequality to $\xi$, it is seen that
  \eqref{eq:DKW-modified} follows.
\end{proof}

\begin{proof}[Proof of Theorem \ref{thm:lp-sequential}]
  By Proposition \ref{prop:equivalence}, under true $H_i$, $s_i\sim
  \vf\pt\tp \vf\phi$, which is continuous and positive on $\CI$.  As a
  result, $\{V(t-)/\vf\pt\tp \vf\phi(t), \, \CF_t,\ t\in \CI\}$
  is a left-continuous backward martingale.

  Fix $c\in \CI$.  By Lemma \ref{lemma:tau}, $\tau_c$ is a stopping
  time of $\{\CF_t, t\in (\inf\CI, c]\}$ with
  $\tau_c\ge t_0>\inf\CI$.  Then $\vf\pt\tp\vf\phi(\tau_c)>0$ and
  $V_n(\tau_c-)/[\vf\pt\tp \vf\phi(\tau_c)]$ is well-defined.  
  By the optional sampling theorem  (cf.\ \cite{karatzas:shreve:91},
  Ch.~1, Thm~3.22),
  \begin{align*}
    E\Sbr{
      \frac{V_n(\tau_c-)}{\vf\pt\tp \vf\phi(\tau_c)}
    } =  
    E\Sbr{
      \frac{V_n(c-)}{\vf\pt\tp \vf\phi(c)}
    } = (1-\pa)n.
  \end{align*}
  Let $c\uto\sup\CI$.  By Lemma \ref{lemma:tau0}, $\tau_c\uto \tau$.
  Because $V_n(\tau_c-)\uto V_n(\tau-)\le n$, $\phi_k(\tau_c)\uto
  \phi_k(\tau)$ and $\vf\pt \tp\vf\phi(\tau_c)\ge \vf\pt\tp
  \vf\phi(t_0)>0$, by dominated convergence,
  \begin{align}\label{eq:optional}
    E\Sbr{
      \frac{V_n(\tau-)}{\vf\pt\tp \vf\phi(\tau)}
    } = (1-\pa)n.
  \end{align}
    
  On the other hand, because $\PD$ is continuous, by Lemma
  \ref{lemma:empirical}, with $\Gamma_n$ defined
  as in \eqref{eq:Gamma},
  \begin{align*}
    &
    E\Sbr{
      \frac{V_n(\tau)}{R_n(\tau)\vee 1}
    } \\
    =\
    &
    E\Sbr{\frac{V_n(\tau-)}{R_n(\tau)\vee 1}}
    +E\Sbr{\frac{V_n(\tau)-V_n(\tau-)}{R_n(\tau)\vee 1}}
    \\
    \le\
    &
    E\Sbr{\frac{V_n(\tau-)}{R_n(\tau)\vee 1}
      \ \vline\ \Gamma_n} P(\Gamma_n)
    +P(\Gamma_n^c) +
    E\Sbr{\frac{V_n(\tau)-V_n(\tau-)}{R_n(\tau)\vee 1}}.
  \end{align*}

  From \eqref{eq:M-tau}, $M_n(\tau)/\alpha \le
  [R_n(\tau)\vee 1]/n$.  On the other hand, conditional on $\Gamma_n$,
  $M_n(\tau)\ge (1-\pa)\vf\pt\tp \vf\phi(\tau)$.  Thus, by
  \eqref{eq:optional}
  \begin{align*}
    E\Sbr{\frac{V_n(\tau-)}{R_n(\tau)\vee 1}
      \ \vline\ \Gamma_n} P(\Gamma_n)
    &\le
    E\Sbr{\frac{\alpha V_n(\tau-)/n}{(1-\pa) \vf\pt\tp \vf\phi(\tau)}
      \ \vline\ \Gamma_n} P(\Gamma_n) \\
    &
    \le
    E\Sbr{\frac{\alpha V_n(\tau-)/n}{(1-\pa) \vf\pt\tp \vf\phi(\tau)}}
    = \alpha.
  \end{align*}
  By Lemma \ref{lemma:empirical}, $P(\Gamma_n^c) \le (1+|\CT_n|)
  \exp(-2 n\epsilon_n^2)$.  Finally, note that $R_n(\tau)=0$ implies
  $V_n(\tau) - V_n(\tau-)=0$ while $V_n(\tau)-V_n(\tau-)\ge 2$ implies
  at least two true nulls have the same value of $s_i$.  Since $s_i$
  under true nulls are iid with a density, the probability of the
  latter event is 0.  Therefore, $V_n(\tau)-V_n(\tau-)\le \cf{R>0}$
  a.s.  This then finishes the proof.
\end{proof}

We next proof Theorem \ref{thm:FDRg}.  For $n\ge 1$, define
\begin{align*}
  \Gamma_n
  = \Cbr{\vfc\in\Delta: \ \vfc\tp\vf\phi\in \CA_n}.
\end{align*}
For each $r>0$, corresponding to \eqref{eq:constraint-global}, define
\begin{align*}
  \Gamma_r = \Cbr{
    \begin{array}{c}
      \vfc\in\Delta: \vfc\tp \vf\phi(t) \le \PD(t)+r,
      \\[.5ex]
      \PD(t_2) - \PD(t_1) \ge \vfc\tp[\vf\phi(t_2)-\vf\phi(t_1)] - r,
      \ t_1\le t_2
    \end{array}
  }
\end{align*}
Both $\Gamma_n$ and $\Gamma_r$ are nonempty since they contain 0.  It
is not hard to see that $\Gamma_n$ and $\Gamma_r$ are convex and
closed, with $\Gamma_r$ being increasing and $\Gamma_0 = \bigcap_{r>0}
\Gamma_r$.  Also, whereas $\Gamma_n$ are random, $\Gamma_r$ are
nonrandom.

Observe that each $t\in \CI$,
\begin{gather}  \label{eq:MmGamma}
  M_n(t) = \sup\{\vfc \tp \vf\phi(t): \vfc\in\Gamma_n\},
  \quad
  m(t) = \sup\{\vfc\tp \vf\phi(t):\ \vfc\in\Gamma_0\}.
\end{gather}
Because $\Gamma_n$ is compact, there is a random $\vfc(t)\in
\Gamma_n$, such that
\begin{gather}
  M_n(t) = \vfc(t)\tp \vf\phi(t).  \label{eq:M-c-t}
\end{gather}

As commented after Theorem \ref{thm:FDRg}, we need to get $M_n\to m$.
One way to do this is to first get $\Gamma_n\to\Gamma_0$, which is
formalized below.
\begin{lemma}\label{lemma:Gamma-conv}
  Let $r>0$.  Then under the conditions of Theorem \ref{thm:FDRg},
  $P(\Gamma_0 \subset \Gamma_n \subset \Gamma_r) \to 1$.
\end{lemma}
\begin{proof}
  By the assumptions, $\PD(t)$ is continuous.  Let
  \begin{align*}
    E_n = \Cbr{\sup_t \Abs{\EP_n(t) - \PD(t)} \le  \epsilon_n/2}.
  \end{align*}
  Then, as in the proof of Lemma \ref{lemma:empirical}, for $n\ge 1$,
  as long as $\exp(-n\epsilon_n^2/2)\le 1/2$, $P(E_n^c) \le
  2\exp\Cbr{-n\epsilon_n^2/2}$.  It is not hard to see that $E_n$
  implies $\Gamma_0 \subset\Gamma_n$.  As $n\epsilon_n^2\toi$,
  $P(\Gamma_0 \subset \Gamma_n) \ge P(E_n) \to 1$.

  Since $\vfc\tp\vf\phi$ is supported by and strictly increasing in
  $\CI$, almost surely, as $n\toi$, the set of $s_i$ under true nulls
  is increasingly dense in $\CI$, and thus so is $S_n=\{\eno s n\}$.
  Because $\phi_k$ and $Q$ are continuous distribution functions, they
  are equicontinuous.  Given $r>0$, fix $C>0$ and $\delta>0$, such
  that 
  \begin{gather*}
    \max_k [\phi_k(-C)+1-\phi_k(C)] + \PD(-C)+1-\PD(C)<r, \\
    \max_k |\phi_k(s)-\phi_k(t)| + |\PD(s)-\PD(t)| < r,
    \ \text{if}\ |s-t| < \delta.
  \end{gather*}

  Let $E_n'=\{\delta(S_n, [-C,C])+\delta(\CT_n, [-C,C])<\delta\}$.
  Conditional on $E_n\cap E_n'$, if $t\in [-C,C]$, then
  $|\PD(t)-\EP_n(t)| \le \epsilon_n$ 
  and there is $s_i$ with $|t-s_i|<\delta$.  Let $\vfc\in
  \Gamma_n$.  By $c_k\ge 0$, $c_1+\cdots+c_L\le 1$ and $\vfc\tp
  \vf\phi(s_i) \le \EP_n(s_i)+\epsilon_n$,
  \begin{align*}
    \vfc\tp\vf\phi(t)
    &
    \le \vfc\tp\vf\phi(s_i) + \max_k
    |\phi_k(t)-\phi_k(s_i)| \\
    &
    \le \EP_n(s_i) + \epsilon_n + r \\
    &
    \le \PD(s_i) + 2\epsilon_n + r \\
    &
    < \PD(t) + 2\epsilon_n + 2r.
  \end{align*}
  If $t\le -C$, then $\vfc\tp\vf\phi(t) \le \max \phi_k(-C) \le r
  \le \PD(t)+r$.  If $t\ge C$, then $\vfc\tp\vf\phi(t) \le 1 \le
  \PD(t)+r$.  In any case, $\vfc\tp\vf\phi(t) \le
  \PD(t)+2\epsilon_n + 2r$.

  Similarly, for $t_1<t_2$, it can be shown that
  $\vfc\tp [\vf\phi(t_2) - \vf\phi(t_1)] < \PD(t_2) - \PD(t_1) +
  3\epsilon_n + 4r$.  As a result, $\vfc\in \Gamma_\sigma$, with
  $\sigma = 3\epsilon_n+4r$.  Then $E_n\cap E_n'\subset \{\Gamma_n
  \subset \Gamma_\sigma\}$.  Because $\epsilon_n\to 0$, $P(E_n\cap
  E_n')\to 1$ and $r$ is arbitrary, the proof is complete.
\end{proof}

\begin{lemma} \label{lemma:M-conv}
  Suppose $\pa < 1$.  Then, under the conditions of Theorem
  \ref{thm:FDRg}, as $n\toi$, $P(M_n\in C(\Reals)$\/$) \to 1$ and $\sup
  |M_n-m|\convP 0$.  Also, $m\in C(\Reals)$.
\end{lemma}
\begin{proof}
  Because each $\phi_k$ is bounded, nondecreasing and continuous,
  $\vf\phi$ is uniformly continuous on $\Reals$.  Since $\Gamma_n$ is
  compact, $\vfc\tp\vf\phi(t)$, $\vfc\in\Gamma_n$ as a family of
  functions in $t$ are equicontinuous and uniformly bounded.  It
  follows that $M_n\in C(\Reals)$.  Likewise, since $\Gamma_0$ is
  compact, $m\in C(\Reals)$.

  Given $\sigma>0$, since $\Gamma_r$ is compact and $\Gamma_r\dto
  \Gamma_0$ as $r\dto 0$, there is $r>0$ such that for all
  $\vfc\in\Gamma_r$, $d(\vfc, \Gamma_0)<\sigma$.  Conditional on
  $\Gamma_0 \subset\Gamma_n$, by \eqref{eq:MmGamma}, $m(t)\le M_n(t)$
  for all $t$.  On the other hand, conditional on $\Gamma_n
  \subset\Gamma_r$, for any $t$, there is $\vfc_0(t)\in
  \Gamma_0$ such that $|\vfc(t) - \vfc_0(t)|\le \sigma$, where
  $\vfc(t)$ is defined as in \eqref{eq:M-c-t}.  Then
  \begin{align*}
    &
    |M_n(t) - \vfc_0(t)\tp \vf\phi(t)|
    \le |\vfc(t) - \vfc_0(t)||\vf\phi(t)| \le \sqrt{L}\sigma \\
    \implies\
    &
    M_n(t) \le \vfc_0(t)\tp \vf\phi(t) + \sqrt{L}\sigma \le m(t) +
    \sqrt{L} \sigma.
  \end{align*}
  Thus, $\{\Gamma_0 \subset \Gamma_n \subset\Gamma_r\}\subset \{0
  \le M_n(t)-m(t) \le \sqrt{L}\sigma$ all $t\}$.  Because $\sigma$ is
  arbitrary, by Lemma \ref{lemma:Gamma-conv}, $\sup |M_n - m|\convP
  0$.
\end{proof}

\begin{proof}[Proof of Theorem \ref{thm:FDRg}]
  The proof follows closely the one in \cite{genovese:was:02}.  By
  Assumption~A and the continuity of $m$ and $Q$, for any
  $0<\epsilon\ll t_*-t_0$,
  \begin{align*}
    \delta = \min\Cbr{
      \inf_{t\in (t_0+\epsilon, t_*-\epsilon)} [\alpha \PD(t) - m(t)], \
      \inf_{t>t_*+\epsilon} [m(t) - \alpha \PD(t)]
    }>0.
  \end{align*}

  Let $\PD_n(t) = [R_n(t)\vee 1]/n$.  As $n\toi$, because $\sup
  |\PD_n-\PD| \convP 0$ and $\sup |M_n-m|\convP 0$, the
  probability that
  \begin{align*}
    \min\Cbr{
      \inf_{t\in (t_0+\epsilon, t_*-\epsilon)} [\alpha \PD_n(t) -
      M_n(t)],\ \inf_{t>t_*+\epsilon}[M_n(t) - \alpha \PD_n(t)]
    }\ge \delta/2
  \end{align*}
  tends to 1, implying $P(|\tau - t_*|\le\epsilon)\to 1$.
  Therefore, $\tau\convP t_*$, which leads to the last claim of the
  theorem.  Since $t_*>\inf\CI$ and $\vfc\tp\vf\phi(t)$ is strictly
  increasing, $\PD(t_*)\ge (1-\pa)\vf\pt\tp\vf\phi(t_*)>0$.  By the Week
  Law of Large Numbers and dominated convergence,
  \begin{align*}
    \fdr = E\Sbr{
      \frac{V_n(\tau)/n}{\PD_n(\tau)}
    } \to \frac{(1-\pa)\vf\pt\tp\vf\phi(t_*)}{\PD(t_*)} \le
    \frac{m(t_*)}{\PD(t_*)} = \alpha,
  \end{align*}
  where the last equality is due to the continuity of $m$ and
  $\PD$ at $t_*$.
\end{proof}

\section{Proofs for Section \ref{sec:mle}}
\label{sec:append-mle}
We need two lemmas for the proof of Proposition \ref{prop:mle}.
\begin{lemma} \label{lemma:mle-S}
  Suppose $\eno f L$ are linearly independent.  Then $S$ in
  \eqref{eq:mle-S} is a convex compact set.
\end{lemma}
\begin{proof}
  It is easy to see that $S$ is convex and closed, so it suffices
  to show $S$ is bounded.  Suppose there are $\vfc_l\in S$ with
  $|\vfc_l| \toi$ as $l\toi$.  Since $c_{l1} +\cdots+
  c_{lk}=1$, this implies $\max_k c_{lk}\toi$ and $\min_k c_{lk}\to
  -\infty$.  There is a subsequence of $\vfc_l$ and a partition
  of $\Theta$ into $\{\theta_{i_1}, \ldots, \theta_{i_r}\}$ and
  $\{\theta_{j_1}, \ldots, \theta_{j_t}\}$, with $r>0$, $t>0$ and
  $r+t=L$, such that $c_{li_s}\ge 0$ and $c_{lj_s}<0$ for each
  $\vfc_l$ in the subsequence.  Without loss of generality, assume
  $c_{li}\ge 0$ for $i=1,\ldots, r$ and $c_{li}<0$ for $i=r+1, \ldots,
  L$.  Denote $d_{li} = -c_{l,r+i}$ for $i=1,\ldots, t$.  Then for
  every $x$,
  \begin{gather*}
    \sum_{k=1}^r c_{lk} f_k(x) \ge \sum_{k=1}^t d_{lk} f_{r+k}(x), \ \
    \sum_{k=1}^r c_{lk} =1+M_l, \text{ with } M_l=\sum_{k=1}^t d_{lk}.
  \end{gather*}
  Divide both sides of the inequality by $M_l$ and let $l\toi$.
  Since $M_l\toi$, there is a sequence of $l$ along which
  $(c_{l1}, \ldots, c_{lr})\tp/M_l$ and $(d_{l1}, 
  \ldots, d_{lt})\tp/M_l$ have limits, say $(\eno u r)\tp$
  and $(\eno v t)\tp$.  Then
  \begin{gather*}
    \sum_{k=1}^r u_k f_k(x) \ge \sum_{k=1}^t v_k f_{r+k}(x), \text{
      all } x.
  \end{gather*}
  It is easy to see that $u_k\ge 0$, $v_k\ge 0$ and $\sum u_k=\sum
  v_k=1$.  Because the integrals of both sides are equal to 1,
  equality must hold.  As a result, $\eno f L$ are linearly
  dependent, which is a contradiction.  
\end{proof}

\begin{lemma}\label{lemma:mle-l}
  Suppose $\int q|\ln f_k|<\infty$ for all $k$.
  \paragraph{1)} For $\vfc\in S^o$ and $r>0$, if $\vfc+\vfv\in S$
  $\forall\vfv\in B(\vf0, r)$ with $\sum v_k=0$, then
  \begin{align} \label{eq:f-lowbound}
    \vfc\tp \vff(x) \ge r[M_f(x) - m_f(x)], \quad\text{all } x,
  \end{align}
  where $M_f(x) = \max_k f_k(x)$ and $m_f(x) = \min_k f_k(x)$.
  \paragraph{2)} For any $\vfc\in S^o$, $\ln(\vfc\tp\vff)\in L^1(\PD)$.
  \paragraph{3)} Let $\ell(\vfc) := \int q\ln(\vfc\tp\vff)$.  Then
  $\ell \in C(S^o)$.
  \paragraph{4)} For any $\vfc\in S^o$ and $x$, $\vfc\tp\vff(x)=0 \iff
  \vff(x)=\vf0$.
  \paragraph{5)} If $\eno f L$ are linearly independent, then
  $\ell$ is strictly concave in $S^o$.
\end{lemma}
\begin{proof}
  1) For any $\vfv\in B(\vf0,r)$ with $\sum v_k=0$, by $(\vfc +
  \vfv)\tp \vff(x) \ge 0$, $\vfc\tp \vff(x) \ge -\sum v_k f_k(x)$.
  Let $v_k = -r$ if $k=\min\{i: f_i(x) = M_f(x)\}$, $v_k=r$ if $k =
  \min\{i: f_i(x) = m_f(x)\}$, and $v_k=0$ otherwise.  Then
  \eqref{eq:f-lowbound} follows.

  2) Let $t^+ = t\vee 0$ and $t^-=(-t)\vee 0$.
  By Lemma \ref{lemma:mle-S}, $\sum c_k^-$ and $\sum
  c_k^+ = \sum c_k^-+1$ are bounded on $S$.  Fix $\lambda\in (0,1)$
  such that $(1-\lambda) \sum c_k^- \le \lambda/2$ on $S$.
  If $M_f(x) > m_f(x)/\lambda$, then by \eqref{eq:f-lowbound},
  $\vfc\tp\vff(x) \ge r(\lambda^{-1}-1) m_f(x)$.  If $M_f(x)\le
  m_f(x)/\lambda$, then 
  \begin{align*}
    \vfc\tp \vff(x)&
    = \sum c_k^+f_k(x) - \sum c_k^-f_k(x) \\
    &
    \ge \sum c_k^+ m_f - \sum c_k^- M_f \\
    &
    \ge \Grp{\sum c_k^+ - \nth\lambda\sum c_k^-} m_f\\
    &
    =\Sbr{1-\Grp{\nth\lambda-1}\sum c_k^-} m_f \ge m_f/2.
  \end{align*}
  Thus, there is a constant $\kappa>0$ such that
  $\vfc\tp \vff(x)\ge \kappa (r\wedge 1) m_f$. 

  On the other hand, $\vfc\tp\vff(x) \le \sum c_k^+ M_f(x) \le \kappa'
  M_f(x)$, where $\kappa'<\infty$ is another constant.  As a result,
  \begin{align} \label{eq:mle-bound}
    |\ln [\vfc\tp\vff(x)]| \le \max\Grp{
      \Abs{\ln [\kappa' M_f(x)]},\,
      \Abs{\ln [\kappa (r\wedge 1) m_f(x)]}
    }
  \end{align}
  Then by $\ln f_k\in L^1(\PD)$, $\ln (\vfc\tp\vff)\in L^1(\PD)$.

  3) Follows from \eqref{eq:mle-bound} and dominated convergence.

  4) If $\vfc\tp\vff(x)=0$, then by \eqref{eq:f-lowbound},
  $M_f(x)=m_f(x)$ and hence $f_k(x)$ are all equal.  As a result,
  $f_k(x) = \vfc\tp\vff(x) = 0$.

  5) For $\vfc_1$, $\vfc_2\in S^o$ and $\theta\in (0,1)$, since $S^o$
  is convex, $\vfc:=(1-\theta)\vfc_1 + \theta\vfc_2 \in S^o$. 
  Because $\ln z$ is strictly concave on $(0,\infty)$,
  $(1-\theta)\ell(\vfc_1) + \theta\ell(\vfc_2) \le \ell(\vfc)$, with
  ``$=$''$\iff$ $\vfc_1\tp \vff(x) = \vfc_2\tp \vff(x)$ for $x$ with
  $q(x)>0$.  On the other hand, if $q(x)=0$, then $\vfc\tp\vff(x)=0$
  and by 4), $\vff(x)=\vf0$.  Therefore, ``$=$'' implies
  $\vfc_1\tp\vff(x)=\vfc_2\tp \vff(x)$ for all $x$.  Since $f_k$ are
  linearly independent, it follows that ``$=$'' $\iff \vfc_1 =
  \vfc_2$.  Therefore, $\ell$ is strictly concave.
\end{proof}

\begin{proof}[Proof of Proposition \ref{prop:mle}]
  By Lemma \ref{lemma:mle-l}, for $\vfc \in S^o$ and $X\sim\PD$, $\ln
  [\vfc\tp \vff(X)]\in L^1$, so by the Weak Law of Large Numbers,
  as $n\toi$, $n^{-1}\sum_{i=1}^n \ln [\vfc\tp \vff(X_i)]\convP
  \ell(\vfc)$.  Since $S$ is compact and $\ell$ is continuous and
  strictly 
  concave on $S^o$, by standard argument, if $\ell$ has a maximum
  point in $S^o$, then the point is unique and $\hat{\vf\pt}_n$
  converges in probability to it.  Thus, to finish the proof, it
  suffices to show that $\vf\pt$ is the maximum point of $\ell(\vfc)$
  if and only if $\int \rho f_k=1$.

  Let $\pi$ be the map $\vfc\to (\eno c {L-1})\tp$ and $d_k(x) =
  f_k(x) - f_L(x)$, $k<L$.  Since $c_1+\cdots+c_L=1$ for $\vfc\in S$,
  then
  \begin{align*}
    \vfc\tp \vff(x) = f_L(x)+\sum_{k=1}^{L-1}
    c_k[f_k(x)-f_L(x)] = f_L(x) + \pi(\vfc)\tp \vfd(x).
  \end{align*}
  Denote $h(\vfu,x) = f_L(x) + \vfu\tp\vfd(x)$ and $H(\vfu) = \int
  q(x) \ln h(\vfu,x)\,dx$.  Then $\ell(\vfc) = H(\pi(\vfc))$.  Since
  $\ell$ is strictly concave in $S^o$, so is $H$ on $\Gamma^o$, with
  $\Gamma= \pi(S)=\{\vfu: f_L + \vfu\tp \vfd\ge 0\}$.  Note that $\pi:
  S\to \Gamma$ is bijective with $\pi^{-1}(\vfu) = (\vfu, 1-\sum u_k)$
  and $\pi(S^o) = \Gamma^o$.  It remains to be seen that $H$ is
  differentiable in $\Gamma^o$, with
  \begin{align*} 
    \frac{\partial H(\vfu)}{\partial u_k}
    = \int \frac{q(x)d_k(x)}{h(\vfu,x)}\,dx, \quad
    k=1,\ldots,L-1.
  \end{align*}
  Once this obtains, by the strict concavity of $H$ and
  $(\eno\pt{L-1})\tp\in\Gamma^o$,
  \begin{align*}
    &
    \text{$\vf\pt$ is the maximum point of $\ell$} \\
    \iff&
    \text{$(\eno\pt{L-1})$ is the maximum point of $H$} \\
    \iff&
    \int \frac{q(x)d_k(x)}{\vf\pt\tp\vff(x)}\,dx = 0 \\
    \iff&
    \int [1-\pa + \pa\rho(x)][f_k(x) - f_L(x)]\,dx = 0 \\
    \stackrel{\rm (a)}{\iff}&
    \int \rho f_k = \int \rho f_L \stackrel{\rm (b)}{\iff}
    \int\rho f_k =1 \text{ all $k$},
  \end{align*}
  where (a) is due to $\pa>0$ and $\int f_k = 1$ and (b) is due to the
  fact that $\int\rho f_k$ being all equal implies each being equal to
  $\int \rho \vf\pt\tp \vff = 1$.

  Given $\vfu\in\Gamma^o$, fix $r>0$ such that $B(\vfu, 2r)\subset
  \Gamma^o$. It is not hard to see that there is $\sigma>0$, such
  that for any $\vfv\in B(\vfu, r)$, $\pi^{-1}(\vfv) + \vf w \in S^o$,
  $\forall \vf w\in B(\vf0, 2\sigma)$ with $\sum w_k = 0$.  Then by
  \eqref{eq:f-lowbound},
  \begin{align} \label{eq:f-lowbound2}
    h(\vfv, x) \ge \sigma[M_f(x)-m_f(x)], \text{ all } \vfv\in
    B(\vfu,r) \text{ and } x.
  \end{align}

 For $x$ with $q(x)>0$ by Lemma \ref{lemma:mle-l}, $h(\vfu+\vfv,x)>0$,
 $\forall \vfv\in B(\vf0, r)$.  Therefore, $\ln [h(\vfu+\vfv, x) /
 h(\vfu, x)]$ is well-defined and by Taylor's expansion, 
  \begin{align*}
    \ln \frac{h(\vfu+\vfv, x)}{h(\vfu, x)}
    = \sum_{k=1}^{L-1} \Sbr{
      \frac{d_k(x) v_k}{h(\vfu,x)} - 
      \frac{d_k(x)^2 v_k^2}{2h(\vfu + z\vfv,x)^2}
    }
  \end{align*}
  for some $z=z(\vfv, x)\in [0,1]$.  As $\vfu + z\vfv\in
  B(\vfu,r)$, by Lemma \ref{lemma:mle-l}, $h(\vfu+z\vfv,x)>0$ and
  by \eqref{eq:f-lowbound2}, $h(\vfu + z\vfv,x) \ge
  \sigma [M_f(x) - m_f(x)]$.  On the other hand, $|d_k(x)| \le
  M_f(x)-m_k(x)$.  Thus $|d_k(x)/h(\vfu + z\vfv, x)|\le
  1/\sigma$.  Likewise, $|d_k(x)/h(\vfu)|\le 1/\sigma$.  As a result,
  \begin{align*}
    H(\vfu+\vfv) - H(\vfu)
    &
    = \int q(x) \ln \frac{h(\vfu+\vfv, x)}{h(\vfu, x)}\,dx \\
    &
    = \sum_{k=1}^{L-1} v_k \int \frac{q(x) d_k(x)}{h(\vfu,x)}\,dx +
    O(|\vfv|^2),
  \end{align*}
  which finishes the proof.
\end{proof}

\section{Proofs for Section \ref{sec:simul}}
\label{sec:append-numeric}
\begin{proof}[Proof of (\ref{eq:lp-seq-modify})]
Recall that the overall distribution under true nulls is $\sum_{k=1}^L
\pt_k \PT_k$ and the distribution of $\eno X n$ is
$\PD = (1-\pa)\vf\pt\tp\vf\PT + \pa\PA$.  Then $(1-\pa)\vf\pt\tp
\vf\PT(X_i) \le \PD(X_i)$.  By the assumption, $\PD$ is continuous,
which implies that $\PD(X_i)$ are iid $\sim \unif(0,1)$.  Then for the
rank statistics 
$X\Sb 1\le X\Sb 2 \le \cdots \le X\Sb n$,
\begin{align*}
  \Grp{\PD(X\Sb k),\,1\le k\le n}
  \sim
  \Grp{
    \frac{\xi_1+\cdots+\xi_k}{\xi_1+\cdots+\xi_{n+1}},
    \, 1\le k\le n
  },
\end{align*}
where $\eno \xi {n+1}$ are iid with density $e^{-x}\cf{x\ge 0}$.  By
exponential inequality, for $\beta\in (0,1)$,
$P(\xi_1+\cdots+\xi_{n+1}<\beta(n+1)) \le (\beta e^{1-\beta})^{n+1}$.
Therefore, for each $k\le a_n$,
\begin{align*}
  &
  P\Grp{\PD(X\Sb i) \le \bar\Gamma^*(1/n; i, 1/\beta)} \\
  =\ &
  P\Grp{
    \frac{\sum_{i=1}^k \xi_i}{\sum_{i=1}^{n+1} \xi_i}
    \le (1/n) g^*(1/n; i, 1/\beta)
  } \\
  \ge\ &
  P\Grp{
    \sum_{i=1}^k \xi_i
    \le \beta g^*(1/n; i, 1/\beta),\
    \sum_{i=1}^{n+1} \xi_i \ge \beta n
  } \\
  \ge\ &
  P\Grp{
    \nth\beta\sum_{i=1}^k \xi_i
    \le  g^*(1/n; i, 1/\beta)}
  -P\Grp{\sum_{i=1}^{n+1} \xi_i < \beta n
  }
\end{align*}
Because $\beta^{-1}\sum_{i=1}^k \xi_i$ follows the Gamma distribution
with shape parameter $k$ and scale parameter $\beta$, by above
inequalities yield
\begin{align*}
  P\Grp{\PD(X\Sb i) \le \bar\Gamma^*(1/n; i, 1/\beta)} 
  \ge 1 - \nth n - (\beta e^{1-\beta})^{n+1}.
\end{align*}
As a result,
\begin{align*}
  &
  P\Grp{(1-\pa)\vf\pt\tp\vf\PT(X\Sb i) \le
    \bar\Gamma^*(1/n; i, 1/\beta),
    \ \text{all}\ i\le a_n} \\
  \ge\
  &
  P\Grp{
    \PD(X\Sb i) \le \bar\Gamma^*(1/n; i, 1/\beta), \ \text{all}\
    i\le a_n 
  } \\
  \ge\
  &
  1 - a_n \Sbr{\nth n + (\beta e^{1-\beta})^{n+1}}.
\end{align*}
Following the proof of Theorem \ref{thm:lp-sequential}, 
\begin{align*}
  \fdr \le\
  &\alpha + E\Sbr{\frac{\cf{R>0}}{R\vee 1}}\\
  &+
  \underbrace{2 (1+|\CT_n|) \exp(- 2n\epsilon_n^2)+
    a_n \Sbr{\nth n + (\beta e^{1-\beta})^{n+1}}
  }_{r_n}.
\end{align*}
Note $\beta e^{1-\beta}<1$.  With $\epsilon_n = \sqrt{\ln n/n}$ and
$|\CT_n|=\lfloor (\ln n)^2\rfloor$, it is easy to see $r_n\to 0$ as
$n\toi$.  Furthermore, for $a_n=n^{0.2}$, $\beta=0.95$, and $n=5000$,
$r_n \approx 9.64\NE{-3}$.
\end{proof}

\bibliographystyle{acmtrans-ims2}

\end{document}